\documentclass[a4paper,11pt]{amsart}

\usepackage{amsmath,amsopn,amssymb,amsfonts,amsthm}
\usepackage[headings]{fullpage}
\usepackage{xspace}
\usepackage{xcolor}
  \colorlet{myblue}{blue!50!black}
  \colorlet{myred}{red!50!black}
  \colorlet{mygreen}{green!50!black}
\usepackage{hyperref}
  \hypersetup{colorlinks=true,urlcolor=myred,linkcolor=myblue,citecolor=mygreen}
\usepackage[nameinlink,capitalise,noabbrev]{cleveref}
\usepackage{graphicx}
\usepackage{algorithm}
\usepackage{algorithmic}
\usepackage{tikz}
  \usetikzlibrary{shapes,arrows,patterns}
\usepackage{pgfplots}
  \pgfplotsset{compat=1.12}
  \usepgfplotslibrary{fillbetween}
\usepackage{enumitem}
  \setlist{leftmargin=0.33in}
  \setlist[enumerate]{label=(\roman*)}
  \setlist[itemize]{label={--}}
\usepackage{bm}
\usepackage{cite}

\theoremstyle{plain}
\newtheorem{theorem}{Theorem}[section]
\newtheorem{corollary}[theorem]{Corollary}
\newtheorem{lemma}[theorem]{Lemma}
\newtheorem{proposition}[theorem]{Proposition}
\newtheorem{definition}[theorem]{Definition}

\newtheorem{sard}[theorem]{Sard's Theorem}
\newtheorem{nonsmooth-sard}[theorem]{Nonsmooth Sard's Theorem}
\newtheorem{monotonicity}[theorem]{Monotonicity Lemma}
\newtheorem{definable-choice}[theorem]{Definable Choice Lemma}
\newtheorem{curve-selection}[theorem]{Curve Selection Lemma}
\newtheorem{assumption}{Assumption}

\theoremstyle{definition}
\newtheorem{remark}[theorem]{Remark}
\newtheorem{example}[theorem]{Example}

\makeatletter
\renewcommand{\ALG@name}{ESQM}
\renewcommand{\fnum@algorithm}{\fname@algorithm}
\makeatother
\crefname{ALC@unique}{Step}{Steps}

\crefname{prb}{Problem}{Problems}
\creflabelformat{prb}{#2(#1)#3}
\crefname{lem}{Lemma}{Lemmas}

\renewcommand{\geq}{\geqslant}
\renewcommand{\leq}{\leqslant}

\newcommand{\N}{\mathbb{N}}
\newcommand{\R}{\mathbb{R}}
\newcommand{\A}{\mathcal{A}}
\newcommand{\Reg}{\mathcal{A}_\text{reg}}
\newcommand{\Sing}{\mathcal{A}_\text{sing}}
\newcommand{\toto}{\rightrightarrows}
\newcommand{\prtb}{\alpha}
\newcommand{\KKT}{\textsc{kkt}}
\newcommand{\NLP}{\mathcal{P}_\textsc{nlp}}
\newcommand{\PNLP}{\mathcal{P}_\prtb}
\newcommand{\PKKT}{\mathcal{P}_\prtb^\KKT}

\DeclareMathOperator{\dom}{dom}

\DeclareMathOperator{\graph}{graph}
\DeclareMathOperator{\dist}{dist}
\DeclareMathOperator{\co}{co}
\DeclareMathOperator{\spn}{span}
\DeclareMathOperator{\val}{val}
\DeclareMathOperator{\argmin}{arg\,min}

\def\<#1,#2>{\langle #1, #2\rangle}
\newcommand{\mybinom}[3][0.8]{\scalebox{#1}{$\dbinom{#2}{#3}$}}

\newcommand{\ESQM}{\hyperref[algo:esqm]{ESQM}\xspace}

\newcommand{\TheTitle}{Qualification Conditions in Semi-algebraic Programming} 
\newcommand{\TheAuthors}{J. Bolte, A. Hochart, and E. Pauwels}

\title{{\TheTitle}}

\date{March 7, 2018}

\author[J.~Bolte]{J\'er\^ome Bolte}
\address{Toulouse School of Economics, Universit\'e Toulouse 1 Capitole, Toulouse, France.}
\email{jerome.bolte@tse-fr.eu}

\author[A.~Hochart]{Antoine Hochart}
\address{Universidad Adolfo Ib{\'a}{\~n}ez, Santiago, Chile.}
\email{antoine.hochart@gmail.com}

\author[E.~Pauwels]{Edouard Pauwels}
\address{Institut de Recherche en Informatique de Toulouse,
Universit\'e Paul Sabatier, Toulouse, France.}
\email{edouard.pauwels@irit.fr}

\thanks{This work was sponsored by the Air Force Office of Scientific Research,
Air Force Material Command, USAF, under grant number FA9550-15-1-0500.
This work was also partially supported by PGMO.}



\ifpdf
\hypersetup{
  pdftitle={\TheTitle},
  pdfauthor={\TheAuthors}
}
\fi

\begin{document}

\maketitle

\begin{abstract}
  For an arbitrary finite family of semi-algebraic/definable functions, we consider
  the corresponding inequality constraint set and we study qualification conditions for
  perturbations of this set.
  In particular we prove that all positive diagonal perturbations, save perhaps a finite
  number of them, ensure that any point within the feasible set satisfies
  Mangasarian-Fromovitz constraint qualification.
  Using the Milnor-Thom theorem, we provide a bound for the number of singular perturbations
  when the constraints are polynomial functions.
  Examples show that the order of magnitude of our exponential bound is relevant. 
  Our perturbation approach provides a simple protocol to build sequences of ``regular''
  problems approximating an arbitrary  semi-algebraic/definable problem.
  Applications to sequential quadratic programming methods and sum of squares relaxation are provided.
\end{abstract}

\section{Introduction}

Constraint qualification conditions ensure that normal cones are finitely generated by the gradients of
the active constraints.
When considering an optimization problem, this fact immediately provides Lagrange/KKT necessary optimality
conditions which are at the root of most resolution methods (see e.g., \cite{NW06,Ber16}).
Finding settings in which qualification conditions are easy to formulate and easy to verify
is thus of fundamental importance.
In a convex framework, the power of Slater's condition consists in its extreme simplicity:
the resolution of a ``simple'' problem (e.g., finding an interior point), often done directly
or through routine computations, guarantees the regularity of the problem.

In a nonconvex setting, the question becomes much more delicate but the wish is the same:
describing normal cones as gradient-generated cones for deriving KKT conditions
(see e.g., \cite{RW98}).
Contrary to what happens for convex functions, the knowledge of the functions at one point
does not capture enough information about the global geometry to infer well-posedness
everywhere\footnote{Observe though that the local knowledge of a polynomial function implies
the knowledge of the function everywhere.
But, as far as we know, this fact has never given birth to any simple qualification condition.}.
Very smooth and simple problems sa\-tis\-fying all possible natural conditions can generally
present a failure of qualification, for which the normal cone is not generated
by the gradients of the active constraints, and thus KKT conditions cannot apply.
In dimension two a typical failure is a cusp, illustrated in \cref{fig:cusp}
for the constraint set
\begin{equation}
  \label{eq:cusp}
  D = \big\{ (x_1,x_2) \in \R^2 \mid x_1^3 + x_2 \leq 0 , \; x_1^3 - x_2 \leq 0 \big\} .
\end{equation}

\begin{figure}[tb]
  \centering
  \def\scale{0.58}
  \def\axes{
    axis lines=none,
    xmin=-1, xmax=0.75, 
    ymin=-1, ymax=1,
    enlargelimits
  }
  \hspace{\stretch{1}}
  \begin{tikzpicture}
    \begin{axis}[\axes,scale=\scale]
      \addplot[name path=g1,gray!50,densely dashed,domain=-1:0.75] {-x^3};
      \addplot[name path=g2,gray!50,densely dashed,domain=-1:0.75] {x^3};
      \addplot[name path=g11,myblue,thick,domain=-1:0] {-x^3};
      \addplot[name path=g22,myblue,thick,domain=-1:0] {x^3};
      \addplot[pattern=north east lines,pattern color=myblue]fill between[of=g11 and g22];
      \draw[red,fill] (0,0) circle (1.5pt);
      \draw[myblue] (-0.3, -1) node {\small $D=C_0$};
    \end{axis}
  \end{tikzpicture}
  \hspace{\stretch{1}}
  \begin{tikzpicture}
    \begin{axis}[\axes,scale=\scale]
      \addplot[name path=g1,gray!50,densely dashed,domain=-1:0.75] {-x^3-0.1};
      \addplot[name path=g2,gray!50,densely dashed,domain=-1:0.75] {x^3+0.1};
      \addplot[name path=g11,myblue,thick,domain=-1:-(0.1)^(1/3)] {-x^3-0.1};
      \addplot[name path=g22,myblue,thick,domain=-1:-(0.1)^(1/3)] {x^3+0.1};
      \addplot[pattern=north east lines,pattern color=myblue]fill between[of=g11 and g22];
      \draw[myblue] (-0.3, -1) node {\small $C_{-0.1}$};
    \end{axis}
  \end{tikzpicture}
  \hspace{\stretch{1}}
  \begin{tikzpicture}
    \begin{axis}[\axes,scale=\scale]
      \addplot[name path=g1,gray!50,densely dashed,domain=-1:0.75] {-x^3+0.1};
      \addplot[name path=g2,gray!50,densely dashed,domain=-1:0.75] {x^3-0.1};
      \addplot[name path=g11,myblue,thick,domain=-1:(0.1)^(1/3)] {-x^3+0.1};
      \addplot[name path=g22,myblue,thick,domain=-1:(0.1)^(1/3)] {x^3-0.1};
      \addplot[pattern=north east lines,pattern color=myblue]fill between[of=g11 and g22];
      \draw[myblue] (-0.3, -1) node {\small $C_{0.1}$};
    \end{axis}
  \end{tikzpicture}
  \hspace{\stretch{1}}
  \caption{On the left, the constraint set $D$, see \labelcref{eq:cusp}:
    the bullet highlights a cusp, for us a failure of constraint qualification.
    Middle and right: negative and positive perturbations of $D=C_0$ make the cusp disappear.}
\label{fig:cusp}
\end{figure}
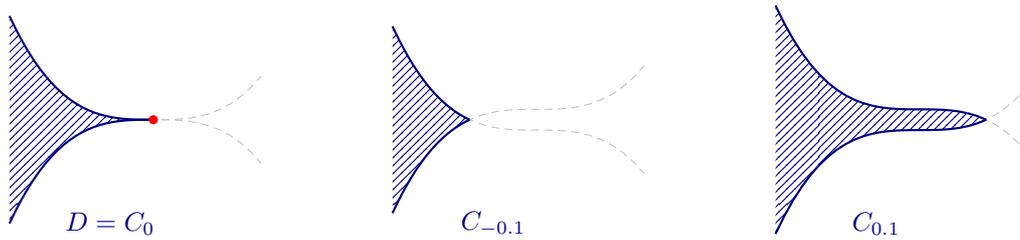

Since in this general setting simple qualification conditions are not available, se\-veral researchers have
considered the problem under the angle of perturbations.
To our knowledge, the first work in this direction was proposed by Spingarn and Rockafellar \cite{RS79}.
Given differentiable functions $g_1,\dots,g_m:\R^n\to \R$ and a constraint set
$C = \{ x \in \R^n \mid g_1(x) \leq 0, \dots, g_m(x) \leq 0\}$, they indeed introduced
the perturbed constraint sets
\[
  C_\mu: = \left[ g_1 \leq \mu_1, \dots, g_m \leq \mu_m \right] \quad
  \text{where $\mu = (\mu_1, \dots, \mu_m) \in \R^m$}
\]
and studied their properties regarding qualification conditions. 
In the sequel, a set $C_{\mu}$ for which qualification conditions hold at each feasible point
is said to be {\em regular}.
Accordingly the corresponding perturbation $\mu$ is called {\em regular}.
When $m = 1$, one obviously recovers the usual definition of a regular value of a function 
(see e.g., Milnor's monograph \cite{Milnor65}), and one guesses that a major role
will be played by Sard-type theorems.
Recall that original Sard's theorem (see e.g., \cite{Milnor65}) expresses that
the regular values of a sufficiently smooth function are generic within $\R^m$.
For $m \geq 1$, the work on  perturbed constraint sets by Spingarn and Rockafellar
\cite{RS79} dealt with the genericity of regular values using  a quite restrictive notion of qualification condition.
Works by Fujiwara \cite{Fuj82}, Scholtes and St\"ohr \cite{SS01} or Nie \cite{Nie14} gave further insights
on different other aspects but with the same type of qualification assumptions.
When the mappings $g_i$ are semi-algebraic (or definable), the application of
definable nonsmooth Sard's theorem of Ioffe \cite{Iof07} yields stronger results
since in that case, regularity exactly corresponds to sets satisfying Mangasarian-Fromovitz
constraint qualification everywhere (see \cref{thm:genericity-definable}).
These aspects are discussed in detail 
in \Cref{sec:pertubed-constraint-sets}.

Genericity results have been the object of a recent revival in connection with
semi-algebraic optimization: see  Bolte, Daniilidis and Lewis \cite{BDL11}, Daniilidis and Pang \cite{pang},
Drusvyatskiy, Ioffe and Lewis \cite{DIL16}, Lee and Pham \cite{LP15}, H\`a and Pham \cite{HP17}.
An original feature of our work is to exploit the fact that genericity is a relative concept.
A property is indeed generic within some given family, but if one considers smaller families,
genericity may no longer hold.
It is therefore important to identify the smallest possible families in order to strengthen
genericity results and to be able to exploit them for improving effective optimization techniques
(e.g., algorithms, homotopy methods).
In this regard, we address in \Cref{sec:finite-singularities} the following two  questions:

\smallskip
\begin{itemize}
    \em
    \item How do we perturb to ensure regularity?
      In other words, how can we build simple problems $(\mathcal{P}_\prtb)_{\prtb \in \R_+}$
      which are regular and whose value, $\val \PNLP$, converges to the one of the original
      problem, $\val ({\mathcal P_0})$?
    \item  Can we go beyond mere genericity and quantify the number of singular (i.e., nonregular)
    values in the polynomial case?
\end{itemize}
\smallskip

Our first result, \`a la Morse-Sard, relies on definability assumptions of the data
(e.g., semi-algebraicity) and provides one-parameter families of regular constraint sets.
This is done by showing that any positive semiline $\R_+ \, v$, with $v \in \R_{++}^m$,
bears only finitely many singular perturbations.
For instance if we let $\bm{\prtb} := (\prtb, \dots, \prtb)$, the sets
$( C_{\bm{\prtb}} )_{\prtb \in \R_+}$ are regular for all $\prtb$ positive small enough,
see  \cref{fig:cusp} for an illustration.
When some of the constraint functions are convex, our approach is considerably simplified:
we show indeed that a ``partial Slater's condition'' allows to restrict the perturbation
approach to nonconvex functions.

The strength of our results is well conveyed by the following general approximation fact:
for any objective $f$ and for $\prtb$ small enough, we are able to build explicit
well-posed problems
\begin{equation*}
  \tag{$\mathcal{P}_{\prtb}$}
  \text{minimize } f(x) \quad \text{s.t.} \enspace x \in C_{\bm{\prtb}}
\end{equation*}
which satisfy, under mild conditions,
$\lim _{\prtb \to 0^+} \val (\mathcal{P}_{\prtb}) = \val (\mathcal{P}_0)$.
Our approach opens the way to continuation methods (see \cite{homotopy} and references therein) or to more direct diagonal methods
as shown in our last section.

A natural question which immediately emerges is whether it is possible to count
the number of singular perturbations.
When assuming further that the data are polynomial functions whose degree is bounded by $d$,
we show by using Milnor-Thom's theorem that the number of singular values for problems of
the type $(\mathcal{P}_{\prtb})$ is lower than $d (2d-1)^n (2d+1)^m$.
Examples show that in general the bound is indeed exponential, even in the quadratic case
with only one of the $g_i$ being nonconvex.  The worst-case bound described in this work is a rather negative result for semi-algebraic programming, in the sense that
it shows that there 
are instances for which  singular values are so clumped and numerous that perturbation techniques are uneffective. The fact that worst-case instances of general semi-algebraic programming is out of reach of modern  methods is a well-known fact since the pioneering work~\cite{smale89}.
It would be interesting to recast our findings along this perspective. For instance, our results suggest that some constraint sets might have such a complex nature that most 
local methods are inapplicable in practice, even after perturbation. On the other hand, as suggested by real-life 
problems, regular instances are numerous in practice.
This shows the need to understand further the geometric factors or probabilistic priors on the constraints that could make singular values less numerous or at least favorably  distributed. 

In \Cref{sec:applications}, we provide two theoretical algorithmic illustrations of our results. As a general fact, our diagonal perturbation scheme can be used in conjunction with any algorithm whose behavior relies on constraint qualification assumptions. We illustrate this principle with exact semidefinite programming relaxations in polynomial programming, for which well-behaved constructions were proposed for regular problems \cite{DNS06,DNP07,ABM14}. A second application of our general results is given by  a class of sequential quadratic programming methods, SQP for short. SQP methods are widespread in practical applications, see e.g,  \cite{Fle85,fred,Aus13}. 
Convergence analysis of such methods usually requires very strong
qualification conditions in order to handle regularity and infeasibility issues for  minimizing sequences. 
We show how our perturbation results provide a natural and strong tool for convergence analysis in the framework of semi-algebraic optimization.

\section{Regular and singular perturbations of constraint sets}
\label{sec:pertubed-constraint-sets}

\subsection{Notation and definitions}

\subsubsection*{Constraint sets and qualification conditions}

Let us consider the general nonlinear optimization problem
\begin{equation}
  \label[prb]{pb:NLP}
  \tag{$\NLP$}
  \begin{aligned}
    \text{minimize} & \enspace f(x)  \\
    \text{subject to} & \enspace g_1(x) \leq 0 , \dots , g_m(x) \leq 0 ,  \\
    & \enspace h_1(x) = 0 , \dots , h_r(x) = 0 , 
  \end{aligned}
\end{equation}
where $f$, $g_1, \dots, g_m, h_1, \dots, h_r$ are differentiable functions
from $\R^n$ to $\R$.
We denote by
\begin{equation*}
  C \; = \; [ g_1 \leq 0, \dots, g_m \leq 0 ] \; := \;
  \{x \in \R^n \mid g_1(x) \leq 0, \dots, g_m(x) \leq 0 \}
\end{equation*}
the inequality constraint set, and by
\begin{equation*}
  M \; = \;  [ h_1 = 0, \dots, h_r = 0 ] \; := \;
  \{x \in \R^n \mid h_1(x) = 0, \dots, h_r(x) = 0 \}
\end{equation*}
the equality constraint set.
For $x \in C$, we define the set of active constraints by
\[
  I(x) \; := \; \{ 1 \leq i \leq m \mid g_i(x) = 0 \} .
\]
We next recall a standard regularity condition.
\begin{definition}[Mangasarian-Fromovitz constaint qualification]
  \label{def:MFCQ}
  A point $x \in C \cap M$ is said to satisfy the {\em Mangasarian-Fromovitz
  constraint qualification (MFCQ)} if the gradient vectors $\nabla h_j(x)$, $j=1,\dots,r$,
  are linearly independent and there exists $y \in \R^n$ such that
  \begin{equation}
    \label{eq:MFCQ}
    \left\{
      \begin{aligned}
        \<y,\nabla h_j(x)> = 0 , & \quad j = 1,\dots, r , \\
        \<y,\nabla g_i(x)> < 0 , & \quad i \in I(x) .
      \end{aligned}
    \right.
  \end{equation}
  If there is no equality constraint, 
  this condition is then called {\em Arrow-Hurwicz-Uzawa constraint qualification}.
  We say that MFCQ holds throughout $C \cap M$ if it is satisfied at every point in $C \cap M$. 
\end{definition}

\begin{remark}
  By a straightforward application of Hahn-Banach's separation theorem,
  the existence of a vector $y \in \R^n$ satisfying condition \labelcref{eq:MFCQ} is equivalent to
  \[
    \co \, \{\nabla g_i(x) \mid i \in I(x) \} \; \cap \;
    \spn \, \{ \nabla h_j(x) \mid 1 \leq j \leq r \} = \emptyset
  \]
  where $\co X$ denotes the convex hull of any subset $X \subset \R^n$, and
  $\spn X$ its linear span.
  If there is no equality constraint, this characterization simply reads
  \[
    0 \notin \co \, \{\nabla g_i(x) \mid i \in I(x) \} .
  \]
  \label{rmk:Hahn-Banach}
\end{remark}

Let us briefly remind that MFCQ guarantees the existence of Lagrange multipliers at
minimizers of \cref{pb:NLP}:
if a local minimizer $\bar{x}$ of $f$ on $C \cap M$ satisfies MFCQ,
then there exist multipliers $\lambda_1,\dots,\lambda_m \in \R_+ := [0,+\infty)$ and
$\kappa_1,\dots,\kappa_r \in \R$ such that
\begin{equation}
  \label{eq:KKT}
  \left\{
    \begin{aligned}
      & \nabla f(\bar{x}) + \sum_{i=1}^m \lambda_i \nabla g_i(\bar{x}) +
      \sum_{j=1}^r \kappa_j \nabla h_j(\bar{x}) = 0 , \\
      & \lambda_i \, g_i(\bar{x}) = 0 , \quad i = 1,\dots,m .
    \end{aligned}
  \right.
\end{equation}
Any feasible point satisfying these conditions is called a {\em Karush-Kuhn-Tucker (KKT) point}.

\begin{remark}[Clarke regularity and MFCQ]
  \label{rmk:Clarke-regularity}
  A more geometrical way of formulating the existence of Lagrange multipliers consists in
  interpreting the gradients of active constraints as generators of a cone normal to
  the constraint set.
  In the terminology of modern nonsmooth analysis, assuming that there are only inequality constraints
  in \cref{pb:NLP}, this amounts to the normal regularity of the set
  $C = [g_1 \leq 0,\dots, g_m \leq 0]$.
  We next explain this fact.

  Being given a nonempty closed subset $X \subset \R^n$, the {\em Fr\'echet normal cone}
  to $X$ at point $\bar{x} \in X$ is defined by
  \[
    \hat{N}_X(\bar{x}) := 
    \{ v \in \R^n \mid \<v,x-\bar{x}> \leq o(\|x-\bar{x}\|), \; x \in X \} .
  \]
  It is immediate to prove that any solution to \cref{pb:NLP} satisfies
  \begin{equation}
    \label{eq:Fermat}
    \nabla f(x) + \hat{N}_C (x) \ni 0
  \end{equation}
  which suggests to express $\hat{N}_C(x)$ with the initial data $g_1,\dots,g_m$.
  To do so, let us introduce the {\em limiting normal {\rm or} Mordukhovich normal
  cone\footnote{See the pioneering work \cite{morduk}.}} to $X$
  at $\bar{x}$, denoted by $N_X(\bar{x})$ and defined by
  \[
    v \in N_X(\bar{x}) \iff \exists \, x_n \to \bar{x} , \quad
    \exists \, v_n \to v , \quad  v_n \in \hat{N}_X(x_n) .
  \]
  The set $X$ is called {\em regular}  at $\bar{x}$ if
  $\hat{N}_X(\bar{x}) = N_X(\bar{x})$.
  
  By classical results of nonsmooth analysis, if the Mangasarian-Fromovitz constraint
  qualification holds throughout $C$, then $C$ is  regular at every point in $C$,
  see \cite[Th.~6.14]{RW98} or \cite[Th.~7.2.6]{BL06}.
  In addition, we have for all $x \in C$
  \[
    N_{C}(x) = \Big\{ \sum_{i \in I(x)} \lambda_i \; \nabla g_i(x) \mid
    \lambda_i \geq 0, \; i \in I(x) \Big\} ,
  \]
  which combined with \Cref{eq:Fermat} yields the claimed result.
\end{remark}

\subsubsection*{Perturbations of constraint sets}

For $\mu \in \R^m$ and $\nu \in \R^r$, we denote by
\begin{align*}
  C_\mu & := [ g_1 \leq \mu_1, \dots, g_m \leq \mu_m ]
  = \{x \in \R^n \mid g_1(x) \leq \mu_1, \dots, g_m(x) \leq \mu_m \} , \\
  M_\nu & := [ h_1 = \nu_1, \dots, h_r = \nu_r ]
  = \{x \in \R^n \mid h_1(x) = \nu_1, \dots, h_r(x) = \nu_r \} ,
\end{align*}
the perturbed inequality and equality constraint sets of \cref{pb:NLP}, respectively.
Also, we denote by
\[
  \A := \{ (\mu,\nu) \in \R^m \times \R^r \mid C_\mu \cap M_\nu \neq \emptyset \}
\]
the set of {\em admissible perturbations}.

\begin{definition}[Regular/Singular perturbations]
  We say that $(\mu,\nu) \in \mathcal{A}$ is a {\em regular perturbation} if
  the Mangasarian-Fromovitz constraint qualification holds throughout $C_\mu \cap M_\nu$,
  and we denote by $\Reg$ the collection of all regular perturbations:
  \[
    \Reg := \{ (\mu,\nu) \in \mathcal{A} \mid
    \text{MFCQ holds at every } x \in C_\mu \cap M_\nu \} .
  \]
  In contrast, an admissible perturbation $(\mu,\nu) \in \mathcal{A}$ is {\em singular}
  if the Mangasarian-Fromovitz constraint qualification is not satisfied at some point
  of $C_\mu \cap M_\nu$.
  The subset of singular perturbations is given by
  \[
    \Sing := \A \setminus \Reg .
  \]
\end{definition}

Up to an obvious change of definition, we shall use the same notation when there is
no equality constraint.

\subsection{Metric regularity and constraint qualification}

In this subsection, we recall how the Mangasarian-Fromovitz constraint qualification
can be interpreted in terms of metric regularity of some set-valued mapping. 
To that purpose, we gather below some classical notions in  nonsmooth analysis, see  \cite{RW98,Mor06,DR14}.

A set-valued mapping $F: \R^p \toto \R^q$ is a map sending each point of $\R^p$ to
a subset of $\R^q$.
We denote by $\graph F := \{ (x,y) \in \R^p \times \R^q \mid y \in F(x) \}$
the {\em graph} of $F$ and by $\dom F := \{ x \in \R^p \mid F(x) \neq \emptyset \}$
its {\em domain}.

The set-valued mapping $F: \R^p \toto \R^q$ is {\em metrically regular}
at $(\bar{x},\bar{y}) \in \graph F$ if the graph of $F$ is locally closed at $(\bar{x},\bar{y})$
and there exist a positive real number $\kappa$, together with neighborhoods $\mathcal{U}$ and
$\mathcal{V}$ of $\bar{x}$ and $\bar{y}$ respectively, such that
\[
  \dist(x,F^{-1}(y)) \leq \kappa \, \dist(y,F(x))
\]
for all $(x,y) \in \mathcal{U} \times \mathcal{V}$.
Here, $\dist(z,K)$ refers to the distance of any point $z$ of a space endowed
with a norm $\| \cdot \|$ to any subset $K$ of the same space, i.e.,
$\inf_{k \in K} \|k-z\|$.

We now come back to \cref{pb:NLP} and introduce the set-valued mapping
$F: \R^n \toto \R^{m+r}$ defined by
\begin{equation}
  \label{eq:constraint-mapping}
  F(x) \quad=\quad 
  \begin{pmatrix}
    g_1(x) \\ \vdots \\ g_m(x) \\ h_1(x) \\ \vdots \\ h_r(x)
  \end{pmatrix}
  \quad+\quad \R^m_+ \times \{0\}^r
\end{equation}
where $\R_+^m $ is the nonnegative orthant of $\R^m$.
Observe that $(\mu,\nu) \in F(x)$ if and only if $x \in C_\mu \cap M_\nu$.
Also notice that, by continuity of the constraint functions, $\graph F$ is closed.

The following result, due to Robinson, characterizes the points satisfying MFCQ in terms of the mapping $F$.
For a thorough discussion and various proofs, we refer the reader to \cite{Mor06}.  Other approaches are \cite[Ex.~9.44]{RW98}, \cite[Ex.~4F.3]{DR14}
or  \cite[Th.~4.1]{DQZ06} which avoids the use of coderivative calculus.

\begin{theorem}[Robinson]
  The Mangasarian-Fromovitz constraint qualification holds at point $x \in C_\mu \cap M_\nu$
  if and only if the set-valued mapping $F$ defined in~\labelcref{eq:constraint-mapping} is
  metrically regular at $(x,(\mu,\nu))$.
  \label{thm:MFCQ-metric-regularity}
\end{theorem}

\subsection{Genericity of regular perturbations}

Qualification conditions play an important role in the analysis of nonlinear programming
and the convergence of optimization algorithms, yet checking these conditions at
optimal points is hardly possible.
This is why one rather seeks local/global simple geometrical assumptions that
automatically warrant these conditions.
Sard's theorem provides results in this direction: generic equations are well-posed if
the data are smooth enough or well-structured (e.g., analytic).
Viewing constraint sets along this angle and following the pioneering work \cite{RS79},
we establish here various genericity results for regular perturbations.

\subsubsection*{Smooth constraint functions}

The first genericity result we present here concerns {\em linear independence
constraint qualification}, a strong and quite stringent qualification condition
which is often considered in the literature when dealing with ``generic'' instances of
optimization problems (see e.g., \cite{Nie14,LP16}).
This qualification condition requires that the gradients of both the equality constraints
and the active inequality constraints are linearly independent.
Note that it implies in particular MFCQ.

Let us first recall classical Sard's theorem.
For a differentiable map $f: \R^p \to \R^q$, a point $x \in \R^p$ is {\em critical}
if the differential mapping of $f$ at $x$ is not surjective.
A {\em critical value} of $f$ is the image of a critical point. Otherwise, $v$ is said to be {\em regular}.

\begin{sard}[see \cite{Milnor65}]
  \label{thm:Sard}
  Let $f: \R^p \to \R^q$ be a map of class $\mathcal{C}^k$ with $k > \max(0,p-q)$.
  Then the Lebesgue measure of the set of critical values of $f$ is zero.
\end{sard}
As a consequence of Sard's theorem, we deduce that a perturbation of the constraint set
of \cref{pb:NLP} is almost surely regular when the constraint functions are
smooth enough.
\begin{theorem}[{compare with \cite[Th.~1]{RS79}}]
  \label{thm:genericity-smooth}
  Let $g_1,\dots,g_m,h_1,\dots,h_r$ be $\mathcal{C}^k$ constraint functions
  from $\R^n$ to $\R$ with $k > \max(0,n-r)$.
  Then the set of admissible perturbations $(\mu,\nu) \in \R^m \times \R^r$ for which
  the linear independence constraint qualification is not satisfied at every point
  of the set $C_\mu \cap M_\nu$ has Lebesgue measure zero.
  In particular, the set $\Sing$ of singular perturbations has Lebesgue measure zero.
\end{theorem}

\subsubsection*{Definable constraint functions}

The above result can be considerably relaxed by replacing smoothness assumptions by mere definability.
The results on definability and tame geometry that we use hereafter are recalled in \cref{sec:tame-geometry}.

Ioffe showed a nonsmooth version of Sard's theorem for definable set-valued mappings.
In this framework, a vector $\bar{y} \in \R^q$ is a {\em critical value} of
any set-valued mapping $F: \R^p \toto \R^q$ if there exists a point $\bar{x} \in \R^p$
such that $\bar{y} \in F(\bar{x})$ and $F$ is not metrically regular at $(\bar{x},\bar{y})$.

\begin{nonsmooth-sard}[{\cite[Th.~1]{Iof07}}]
  Let $F:\R^p \toto \R^q$ be a definable set-valued mapping with locally closed graph.
  Then the set of critical values of $F$ is a definable set in $\R^q$
  whose dimension is less than $q-1$.
  \label{thm:Sard-definable}
\end{nonsmooth-sard}

Combining this result with \cref{thm:MFCQ-metric-regularity}, we readily get
a geometric description of regular perturbations for \cref{pb:NLP}
when the constraint functions are definable in the same o-minimal structure.
Let us mention that we also use the fact that any definable set $A \subset \R^p$ can be
``stratified'', that is, written as a finite disjoint union of smooth submanifolds of $\R^p$
that fit together in a ``regular'' manner.
This implies in particular that the dimension of $A$, i.e., the largest dimension of such
submanifolds, is strictly lower than $p$ if and only if the complement of $A$ is dense.

\begin{theorem}[Genericity of regular perturbations]
  \label{thm:genericity-definable}
  Let $g_1,\dots,g_m,h_1,\dots,h_r: \R^n \to \R$ be constraint functions that are
  definable in the same o-minimal structure.
  Then the set $\Reg$ (resp., $\Sing$) of regular (resp., singular) perturbations is definable
  in $\R^{m+r}$, and $\Sing$ is a finite union of smooth submanifolds of $\R^{m+r}$ of
  dimension strictly lower than $m+r$.
\end{theorem}

\begin{remark}
  \label{rmk:topology-singularities}
  Note that, in general, the set $\Sing$ of singular perturbations is not closed.
  Consider for instance the semi-algebraic functions defined on $\R^2$ by
  \begin{gather*}
    g_1(x_1, x_2) = \min \left\{ 2 x_1-1, \frac{1}{|x_1|} \right\} - x_2 , \\
    h_1(x_1, x_2)  = \frac{x_1}{1+(x_1)^2} - x_2 , \quad
    h_2(x_1, x_2)  = \frac{x_1}{1+(x_1)^2} + x_2 .
  \end{gather*}
  For every $\nu \in (0,\frac{1}{2})$, the set $[h_1=\nu, h_2=\nu]$ contains two distinct points,
  namely $(\frac{1 \pm \sqrt{1 - 4 \nu^2}}{2 \nu}, 0)$, whereas $[h_1=0, h_2=0] = \{(0,0)\}$.
  Let $\mu_\nu = \frac{1 + \sqrt{1 - 4 \nu^2}}{2 \nu}$.
  One easily checks that, for the constraint set $[g_1 \leq \mu_\nu^{-1}, h_1 = \nu, h_2 = \nu]$
  with $\nu \in (0,\frac{1}{2})$, MFCQ fails at point $(\mu_\nu, 0)$ but it is satisfied at point
  $(\frac{1 - \sqrt{1 - 4 \nu^2}}{2 \nu}, 0)$, where the inequality constraint is not active.
  Hence $(\mu_\nu^{-1}, \nu, \nu)$ is a singular perturbation for all $\nu \in (0,\frac{1}{2})$.
  However, when $\nu = 0$ the singularity disappears and only remains the point $(0,0)$
  at which MFCQ holds.
  In other words, $(0,0,0)$ is regular.
\end{remark}

\subsection{Continuity properties of perturbations}

We investigate below the continuity properties of the perturbed constraint sets and
of the value function of \cref{pb:NLP}.
Recall beforehand that given any set-valued mapping
$F: \R^p \toto \R^q$, the {\em outer limit}, $\limsup_{x \to \bar{x}} F(x) \subset \R^q$,
and the {\em inner limit}, $\liminf_{x \to \bar{x}} F(x) \subset \R^q$, of $F$
at any point $\bar{x} \in \R^p$ are defined respectively by the following:
\begin{align*}
  y \in \limsup_{x \to \bar{x}} F(x) & \iff 
  \exists \, x_n \to \bar{x} , \quad \exists \, y_n \to y , \quad
  \forall \, n \in \N , \quad y_n \in F(x_n) , \\
  y \in \liminf_{x \to \bar{x}} F(x) & \iff 
  \forall \, x_n \to \bar{x} , \quad \exists \, y_n \to y , \quad
  \exists \, n_0 \in \N , \quad \forall \, n \geq n_0 ,
  \quad y_n \in F(x_n) .
\end{align*}
Then we can define the notion of (semi)continuity for set-valued mapping.

\begin{definition}[Semicontinuity of set-valued mappings]
  A set-valued mapping $F: \R^p \toto \R^q$ is {\em outer semicontinuous} (resp.,
  {\em inner semicontinuous}) at $\bar{x} \in \R^p$ if
  \[
    \limsup_{x \to \bar{x}} F(x) \subset F(\bar{x}) \quad \Big( \text{resp.,} \enspace
    \liminf_{x \to \bar{x}} F(x) \supset F(\bar{x}) \Big)  .
  \]
  It is {\em continuous} at $\bar{x}$ if it is both outer and inner semicontinuous.
\end{definition}

A straightforward application of these definitions leads to the  elementary lemma:

\begin{lemma}[Continuity of perturbed sets]
  \label{lem:continuity-constraint-set}
  Let $g_1,\dots,g_m: \R^n \to \R$ be continuous functions.
  Assume that the constraint set $C_0 = [g_1 \leq 0,\dots, g_m \leq 0]$ is nonempty.
  Then the set-valued mapping $\R^m_+ \toto \R^n, \; \mu \mapsto C_\mu$ is continuous at $0$.
  \hfill \qed
\end{lemma}

\begin{remark}
  It is in general necessary to consider nonnegative perturbations
  in order to have continuity at $0$.
  Indeed, for general perturbations, although the inequality constraint set mapping
  $\R^m \toto \R^n, \; \mu \mapsto C_\mu$ is outer semicontinuous at $0$ (this readily
  follows from the continuity of the constraint functions), it is not inner semicontinuous.
  Consider for instance the following constraint set, defined for any $\mu \in \R^2$ by
  \[
    C_\mu = \{x \in \R \mid 1-x^2 \leq \mu_1, \; (x+1)^2 - 4 \leq \mu_2 \} .
  \]
  Check that $C_0 = [-3,-1] \cup \{1\}$.
  However, for all $\mu_1 < 0$ and $\mu_2 < 0$ small enough, we have
  $C_\mu = [-1-\sqrt{4+\mu_2}, -\sqrt{1-\mu_1}] \subset [-3, -1]$.
  Hence $\{1\}$ cannot be in the inner limit of $C_\mu$ as $\mu \to 0$.
  Precisely, we have $\liminf_{\mu \to 0} C_\mu = [-3,-1]$.

  As for the equality constraint set mapping $\R^r \toto \R^n, \; \nu \mapsto M_\nu$,
  it is also clearly outer semicontinuous at $0$ 
  but not inner semicontinuous in general, even when restricting to $\R_+^r$.
  For instance, consider for $\nu \in \R$ the constraint set
  \[
    M_\nu = \{x \in \R \mid 9 x (x^2-1) - 2 \sqrt{3} = \nu \}
  \]
  and check that $M_0 = \{-1 / \sqrt{3}, \; 2/ \sqrt{3}\}$.
  However, for every $\nu > 0$, $M_\nu$ contains a unique point, which converges to $2 / \sqrt{3}$
  as $\nu$ tends to $0$.
  As a consequence, the inner limit of $M_\nu$ at $0$ can only contain this point.
  Studying the situation when $\nu < 0$, one readily see that, actually,
  $\liminf_{\nu \to 0} M_\nu = \{2 / \sqrt{3}\}$.
\end{remark}

We now turn our attention to the behavior of the value function of perturbed problems and we study
 the continuity at $0$ of 
$(\mu,\nu) \mapsto \min \{ f(x) \mid  x \in C_\mu \cap M_\nu \}$.
As a consequence of previous observations: continuity cannot occur in general
when equality constraints are present, and it is  ``necessary" to consider nonnegative perturbations
for the inequalities.
The next result is a classical, see e.g., \cite[Prop.\ 4.4]{BS00}.
In the following, we denote by $\R_{++}^m$  the set of vectors in $\R^m$ with positive entries.

\begin{lemma}[Continuity of the value function]
  \label{lem:continuity-value}
  Let $f, g_1,\dots,g_m: \R^n \to \R$ be continuous functions.
  Assume that the constraint set $C_\mu = [g_1 \leq \mu_1,\dots, g_m \leq \mu_m]$
  is nonempty for $\mu = 0$ and bounded for some positive perturbation $\mu' \in \R_{++}^m$.
  Then the value function $\val: \R_+^m \to \R$ defined by
  $\displaystyle \val(\mu) = \min_{x \in C_\mu} f(x)$
  is continuous at $0$:
  \[
    \min_{x \in C_\mu} f(x) \enspace
    \xrightarrow[\substack{\mu \to 0 \\ \mu \in \R_+^m}]{} 
    \enspace \min_{x \in C_0} f(x) \enspace .
  \]
\end{lemma}

\begin{proof}
  First, since $C_0 \subset C_\mu$ for all $\mu \in \R_+^m$, we have
  $\val(0) \geq \limsup_{\mu \to 0} \val(\mu)$.

  Let $(\mu_n)_{n \in \N}$ be any sequence in $\R_+^m$ converging to $0$ and
  such that $\val(\mu_n)$ converges to some $v \in \R \cup \{-\infty\}$.
  Since $\mu' \in \R_{++}^m$, we may assume without loss of generality that we have,
  for all integers $n$, $\mu' - \mu_n \in \R_+^m$, so that $C_{\mu_n} \subset C_{\mu'}$.
  Let $x^*_n \in \argmin \{ f(x) \mid  x \in C_{\mu_n} \}$, that is, $x^*_n \in C_{\mu_n}$
  and $f(x^*_n) = \val(\mu_n)$.
  Since the sequence $(x^*_n)_{n \in \N}$ lies in the bounded set $C_{\mu'}$, it converges,
  up to an extraction, to some point $x^*$.
  Therefore, we have $v = f(x^*)$ with $x^* \in C_0$ by continuity of $f$ and
  $\mu \mapsto C_\mu$ (\cref{lem:continuity-constraint-set}).
  We deduce that $\val(0) \leq \liminf_{\mu \to 0} \val(\mu)$,
  which concludes the proof.
\end{proof}

Note that, without the compactness assumption, the conclusion of
\cref{lem:continuity-value} does not hold.
Consider for instance the semi-algebraic programming problem
\[
  \text{minimize} \enspace \frac{1+x^2}{1+x^4} \quad
  \text{subject to} \enspace x \in \R , \enspace \frac{x^2}{1+x^4} \leq \prtb .
\]
For all scalars $\prtb > 0$, the value of the problem is $\val(\prtb) = 0$,
whereas for $\prtb = 0$ it is $\val(0) = 1$.

\section{Finiteness of singular diagonal perturbations}
\label{sec:finite-singularities}

\subsection{Geometric aspects of regular perturbations}

Although \cref{thm:genericity-definable} is a satisfying theoretical result,
it does not give any structural information beyond dimension and definability.
In particular it is not clear {\em how the perturbations} should be chosen 
when dealing with concrete optimization problems. 
The following result shows that, under reasonable assumptions, {\em small positive} and
{\em small negative} perturbations $\mu$ of the inequality constraints are always regular,
that is, MFCQ is satisfied at every point in $C_{\mu}$.
For the sake of simplicity, we have chosen to state this result
as well as all the subsequent ones for constraint sets defined only by inequalities.
Nevertheless, they all extend easily to the setting of inequality and equality constraints
(with perturbations applying only to the inequalities), see
\cref{rmk:regular-perturbations}~\labelcref{it:equality-constraints-i} and
\cref{rmk:equality-constraints-ii,rmk:equality-constraints-iii,rmk:equality-constraints-iv,rmk:equality-constraints-v}.

\begin{theorem}[Small regular perturbations]
  \label{thm:regular-local-perturb}
  Let $g_1,\dots$, $g_m: \R^n \to \R$ be differentiable constraint functions that are definable
  in the same o-minimal structure.
  \begin{enumerate}
    \item {\em (Outer regular perturbations).}
      If $C_0 = [g_1 \leq 0, \dots, g_m \leq 0]$ is nonempty, then there exists
      $\varepsilon_0 > 0$ such that $(0,\varepsilon_0)^m \subset \Reg$.
      In other words, for all positive perturbations $\mu \in (0,\varepsilon_0)^m$,
      the Mangasarian-Fromovitz constraint qualification holds throughout
      $C_\mu = [g_1 \leq \mu_1, \dots, g_m \leq \mu_m]$.
      \label{it:outer-perturb}
    \item {\em (Inner regular perturbations).}
      If $[g_1 < 0, \dots, g_m < 0]$ is nonempty, then there exists $\varepsilon_1 > 0$
      such that $(-\varepsilon_1,0)^m \subset \Reg$.
      In other words, for all negative perturbations $\mu \in (-\varepsilon_1,0)^m$,
      the Mangasarian-Fromovitz constraint qualification holds throughout
      $C_\mu = [g_1 \leq \mu_1, \dots, g_m \leq \mu_m]$.
      \label{it:inner-perturb}
  \end{enumerate}
\end{theorem}

\begin{proof}
  We only show \cref{it:outer-perturb}.
  \Cref{it:inner-perturb} follows from very similar arguments.
  Let us first notice that the constraint set mapping
  $\R^m_+ \toto \R^n, \; \mu \mapsto C_\mu$ is a definable mapping. 
  For each $\mu$ in $\R^m_{++}$ we consider the subset $S(\mu)$ of $C_\mu$ consisting of
  the points at which MFCQ is not satisfied.
  Following \cref{rmk:Hahn-Banach}, we have
  \begin{equation}
    \label{eq:singular-point-map}
    S(\mu) = \Big\{ x \in C_\mu \mid 0 \in  \co \, \{\nabla g_i(x) \mid i \in I(x) \} \Big\} .
  \end{equation}
	This extends to a definable set-valued mapping $S: \R^m \toto \R^n, \; \mu \mapsto S(\mu)$
  by setting $S(\mu) = \emptyset$ if $\mu \not \in \R^m_{++}$.

  Towards a contradiction, we assume that $0$ belongs to the closure of $\dom S$.
  Using the \cref{lem:curve-selection}, we obtain a definable $\mathcal{C}^1$ curve
  $[0,1) \to \R^m, \; t \mapsto \mu(t)$ such that $\mu(t) \in \dom S$ for all $t > 0$
  and $\mu(0) = 0$.
	The \cref{lem:monotonicity} combined with the fact that
  $\mu(t) \in \dom S \subset \R^m_{++}$ for $t \in (0,1)$	and $\mu(0) = 0$,
  ensures the existence of $\varepsilon > 0$ such that 
  \begin{equation}
		\label{eq:increase-mu}
    \dot{\mu}_i(t) > 0 , \quad  i = 1,\dots,m ,
    \quad \forall \, t \in (0,\varepsilon) .
  \end{equation}
  
  The set-valued mapping $(0,\varepsilon) \toto \R^n, \; t \mapsto S(\mu(t))$ is definable
  and has nonempty values, hence the \cref{lem:definable-choice} yields the existence of
  a definable curve $x: (0,\varepsilon) \to \R^n$ such that $x(t) \in S(\mu(t))$ for all $t$.
	Shrinking $\varepsilon$ if necessary (using the \cref{lem:monotonicity}) we can assume
  that $x(\cdot)$ is $\mathcal{C}^1$.

  Being given two definable functions $a, b: (0,\varepsilon) \to \R$, we can apply
  once more \hyperref[lem:monotonicity]{Lemma~\labelcref{lem:monotonicity}}
  to see that either $a(t) = b(t)$ or $a(t) > b(t)$
  or $b(t) > a(t)$ for $t$ sufficiently small.
	This implies in particular that there exists a positive real $\varepsilon' \leq \varepsilon$ 
	and a nonempty subset $I \subset \{1,\ldots, m\}$ such that $I(x(t)) = I$
  for all $t \in (0,\varepsilon')$. 
  Indeed, recall that $I(x(t)) = \{1 \leq i \leq m \mid g_i(x(t)) = \mu_i(t) \}$,
  where each pair of functions $(g_i(x(\cdot)),\mu_i(\cdot))$, $i=1,\dots,m$, is definable. 
	Hence $I(x(t))$ stabilizes for $t > 0$ sufficiently small.
  Furthermore, for all $t \in (0, \varepsilon)$, $I(x(t))$ is nonempty because otherwise 
	MFCQ would be satisfied at $x(t)$, which would contradict the fact that $x(t) \in S(\mu(t))$.

  By definition of $S$, for all $t \in (0,\varepsilon')$ there exist coefficients $\lambda_i(t)$
  with $i \in I$ such that
  \[
    \lambda_i(t) \geq 0 , \enspace \forall i \in I , \quad \text{and} \quad
    \sum_{i \in I} \lambda_i(t) = 1 ,
  \]
  and such that 
  \begin{equation}
    \label{eq:convex-combination}
    \sum_{i \in I} \lambda_i(t) \, \nabla g_i(x(t)) = 0 .
  \end{equation}
  Multiplying each member of the above equality by $\dot{x}(t)$, one obtains
  \[
    \sum_{i \in I} \lambda_i(t) \, \< \: \dot{x}(t), \nabla g_i(x(t)) \: > = 0
  \]
  which also writes
  \[
    \sum_{i \in I} \lambda_i(t) \, \frac{d (g_i \circ x)}{dt} (t) = 0 .
  \]
  Since each inequality constraint $g_i$, $i \in I$, is active for all $t \in (0,\varepsilon')$,
  one gets
  \[
    \sum_{i \in I} \lambda_i(t) \, \dot{\mu}_i(t) = 0 , \quad
    t \in (0,\varepsilon') ,
  \]
	a contradiction: indeed \Cref{eq:increase-mu} and the fact that $I \neq \emptyset$ indicate
  that the left-hand side of the latter equality is positive for all $t \in (0,\varepsilon')$.

  For \cref{it:inner-perturb}, it suffices to notice that Slater's condition
  guarantees that the set $C_\mu$ with $\mu \in (-\infty,0)^m$ is nonempty
  for $\mu$ in a neighborhood of $0$.
  The proof follows then arguments similar to those developed
  for \cref{it:outer-perturb}.
\end{proof}

We next discuss some aspects of the hypotheses of \cref{thm:regular-local-perturb}.

\begin{remark}
  \label{rmk:regular-perturbations}
  \begin{enumerate}[label=(\alph*),leftmargin=0pt,itemindent=0.25in,labelsep=3pt]
    \item A similar result cannot be derived for joint perturbations $(\mu,\nu)$ of
      inequality and equality constraint sets.
      Consider for instance the functions defined on $\R^2$ by $g(x_1,x_2) = x_2$ and
      $h(x_1,x_2) = x_2-(x_1)^2$.
      For all perturbations $\mu \in \R$, the constraint set
      $C_\mu \cap M_\mu = [g \leq \mu, h = \mu]$ contains only the point $(0,\mu)$,
      at which MFCQ does not hold since  $\nabla g(0,\mu) = \nabla h(0,\mu) = (0,1).$
      See also \cref{rmk:topology-singularities}.
    \item The conclusion of \cref{thm:regular-local-perturb} still holds for perturbed constraint
      sets of the form $C_\mu \cap M_0$.
      For this, one needs to assume that the equality constraint set $M_0 = [h_1 = 0,\dots,h_r = 0]$
      is defined by definable differentiable functions $h_1,\dots,h_r: \R^n \to \R$ whose gradient
      vectors $\nabla h_j(x)$, $j = 1,\dots,r$, are linearly independent for all $x \in M_0$.
      
      \smallskip
      \begin{small}
        \noindent
        $\left[\right.${\em Sketch of proof.}
        Only a few changes in the proof of the previous theorem are necessary.
        The first one is the definition of the set-valued map $S$ \cref{eq:singular-point-map}, which
        sends perturbation vectors $\mu \in \R^m$ to the set of feasible points where constraint
        qualification conditions are not satisfied.
        In this new setting, it becomes
        \[
          S(\mu) = \Big\{ x \in C_\mu \cap M_0 \mid \co \, \{\nabla g_i(x) \mid i \in I(x) \} \cap 
          \spn \, \{ \nabla h_j(x) \mid 1 \leq j \leq r \} \neq \emptyset \Big\} .
        \]
        The second change is \Cref{eq:convex-combination} which characterizes the failure of
        constraint qualification at point $x(t) \in S(\mu(t))$.
        Since the gradient vectors of the equality constraints are linearly independent
        for all the feasible points, this failure of constraint qualification must come from
        the absence of a vector $y$ satisfying \cref{eq:MFCQ}.
        Hence, following \cref{rmk:Hahn-Banach}, the right-hand side of the equality must be replaced
        by a linear combination of the gradients $\nabla h_j$ at point $x(t)$, that is,
        \Cref{eq:convex-combination} now reads
        \[
          \sum_{i \in I} \lambda_i(t) \, \nabla g_i(x(t)) =
          \sum_{j=1}^r \kappa_j(t) \, \nabla h_j(x(t))
        \]
        for some coefficients $\kappa_j(t) \in \R$, $j = 1, \dots, r$.
        Then the proof proceeds along the same lines.
        In particular, multiplying the right-hand side of the previous equality by $\dot{x}(t)$,
        one obtains
        \[
          \sum_{j=1}^r \kappa_j(t) \, \< \: \dot{x}(t), \nabla h_j(x(t)) \: >
          = \sum_{j=1}^r \kappa_j(t) \, \frac{d (h_j \circ x)}{dt} (t) = 0
        \]
        since each equality constraint $h_j$, $j=1,\dots,r$, is constant on the curve $x$.$\left.\right]$
      \end{small}
      \smallskip
      \label{it:equality-constraints-i}
    \item The use of small perturbation vectors which are not positive (or negative) may
      not remove the absence of MFCQ.
      A simple example is given on $\R^2$ by
      \[
        g_1(x_1,x_2) = (x_1)^2 + (x_2)^2 - 1 \quad \text{and} \quad
        g_2(x_1,x_2) = (x_1-2)^2 + (x_2)^2 - 1 .
      \]
      The sets $[g_1 \leq 0]$ and $[g_2 \leq 0]$ delineate two tangent discs.
      Therefore, MFCQ fails at their contact point $(1,0)$.
      Consider the perturbation path $t \mapsto (\mu_1(t), \mu_2(t)) = (t^2 - 2t, t^2 + 2t)$
      which passes through $(0,0)$ with velocity $(2,-2)$.
      It can be checked that the constraint set $[g_1 \leq \mu_1(t), g_2\leq \mu_2(t)]$ is
      not regular at $(1-t,0)$ for all $t \in (-1,1)$.
    \item Even though definable functions are not the unique class of functions for which
      a theorem similar to \cref{thm:regular-local-perturb} can be derived\footnote{One can think
      for instance of continuous convex functions.}, the definability assumption
      in \cref{thm:regular-local-perturb} cannot be replaced by mere smoothness.
      Many counterexamples can be given, even when $n=1$.
      Consider for instance the strictly increasing $\mathcal{C}^{\infty}$ function
      $g(x) = \int_0^x \exp(-t^{-2}) \, \sin^2(t^{-1}) dt$
      with $x \in \R$, and the set $[g \leq 0] = (-\infty,0]$.
      Obviously the set of regular perturbations $\Reg$ does not contain any segment of
      the form $(0,\varepsilon)$ with $\varepsilon > 0$.
  \end{enumerate}
\end{remark}

We now provide a ``partial perturbation version'' of our main result, which can be proved 
following the lines of  \cref{thm:regular-local-perturb}.
It relies on the assumption that the set defined by the first $p$ inequalities
is regular.

\begin{theorem}[Partial constraint qualification]
  Let $g_1,\dots,g_m: \R^n \to \R$ be differentiable functions that are definable
  in the same o-minimal structure.
  Assume that $C_0 = [g_{1} \leq 0,\dots, g_{m} \leq 0]$ is nonempty and that
  the Mangasarian-Fromovitz constraint qualification holds throughout 
  $[g_{1} \leq 0,\dots, g_{p} \leq 0]$ for some positive integer $p < m$.
  Then there exists $\varepsilon > 0$ such that, for all perturbations
  $\mu_{p+1}, \dots, \mu_m \in (0,\varepsilon)$, MFCQ holds throughout
  $[g_1 \leq 0, \dots, g_p \leq 0, \, g_{p+1} \leq \mu_{p+1}, \dots, g_{m} \leq \mu_{m}]$.
  \hfill \qed
  \label{thm:partial-perturb}
\end{theorem}

\begin{remark}
  \label{rmk:equality-constraints-ii}
  Similarly to \cref{rmk:regular-perturbations}~\labelcref{it:equality-constraints-i},
  \cref{thm:partial-perturb} also holds in the setting of fixed equality
  constraints, in addition to partially perturbed inequality constraints.
\end{remark}

A simple but important corollary is a kind of partial Slater's qualification condition. 
\begin{corollary}[Partial Slater's condition]
  Let $g_1, \dots, g_m: \R^n \to \R$ be differentiable convex functions that are definable
  in the same o-minimal structure.
  Assume that $C_0 \neq \emptyset$ and that $g_i(x_0) < 0$ for some $x_0 \in \R^n$ and all
  $i \in \{1,\dots,p\}$ where $p < m$.
  Then there exists $\varepsilon > 0$ such that, for all perturbations
  $\mu_{p+1}, \dots, \mu_m \in (0,\varepsilon)$, MFCQ holds throughout
  $[g_1 \leq 0, \dots, g_p \leq 0, \, g_{p+1} \leq \mu_{p+1}, \dots, g_{m} \leq \mu_{m}]$.
  \hfill \qed
  \label{cor:partial-slater}
\end{corollary}

\begin{remark}
  \label{rmk:equality-constraints-iii}
  The above result is provided without equality constraints for the sake of simplicity:
  the addition of a finite system of affine constraints\footnote{The linear independence
  assumption of \cref{rmk:regular-perturbations}~\labelcref{it:equality-constraints-i} is
  not necessary in this case.}
  is an easy task.
\end{remark}

To conclude this subsection, let us emphasize that \cref{thm:regular-local-perturb}
applies to several frameworks that are widely spread in practice, including
polynomial, semi-algebraic, real analytic and many other kind of definable constraints.

\subsection{Diagonal perturbations}

When the constraint functions are definable in the same o-minimal structure, then so is
the set of regular (resp., singular) perturbations, since it can be described by
a first-order formula (see also \cref{thm:genericity-definable}).
Thus, a direct application of \cref{thm:regular-local-perturb} leads to
the finiteness of singular perturbations along any direction, and in particular
along the {\em diagonal}, i.e., for perturbations of the form $(\prtb,\dots,\prtb)$
with $\prtb \in \R$.
With a minor abuse of notation, for $\prtb \in \R$, we use the following,
$C_\prtb := [g_1 \leq \prtb,\dots,g_m \leq \prtb]$.

\begin{corollary}
  \label{cor:diagonal-perturb}
  Let $g_1,\dots,g_m: \R^n \to \R$ be differentiable functions that are definable
  in the same o-minimal structure.
  Then, for all except finitely many perturbations $\prtb \in \R$, the Mangasarian-Fromovitz constraint
  qualification holds throughout $C_\prtb = [g_1 \leq \prtb,\dots,g_m \leq \prtb]$.
  The same conclusion holds for
  \[
    [g_1 \leq 0,\dots, g_p \leq 0, g_{p+1} \leq \prtb, \dots, g_{m} \leq \prtb]
  \]
  if MFCQ holds throughout $[g_{1} \leq 0,\dots, g_{p} \leq 0]$ for some $p < m$.
\end{corollary}

\begin{remark}[Equality and inequality constraint]
  \label{rmk:equality-constraints-iv}
  Similarly to \cref{rmk:regular-perturbations}~\labelcref{it:equality-constraints-i}, the above result
  also holds for perturbed constrained sets of the form $C_\alpha \cap M_0$, when $M_0$ is defined by
  definable functions that are differentiable and whose gradients are independent at any point $x \in M_0$.
\end{remark}

\begin{remark}[Diagonal perturbations through nonsmooth Sard's theorem]
  With a slightly stronger regularity assumption, \cref{cor:diagonal-perturb}
  can be seen as the nonsmooth definable Sard-type theorem \cite[Cor.~9]{BDLS07}.
  We next explain this observation.

  When dealing with diagonal perturbations $\prtb \in \R$, the constraint set
  $C_\prtb$ can be represented as the lower level set of a single
  real-extended-valued function.
  Namely, a point $x$ is in $C_\prtb$ if and only if
  \[
    \max_{1 \leq i \leq m} g_i(x) \leq \prtb .
  \]

  Let us define $g = \max_{1 \leq i \leq m} g_i$ and let us assume that the constraint
  functions $g_1,\dots,g_m$ are $\mathcal{C}^1$.
  This implies that the basic chain rule of subdifferential calculus applies to the function
  $g$, see \cite[Th.~10.6]{RW98}.
  Thus we have, for all $x \in \R^n$,
  \begin{equation*}
     \widehat{\partial} g(x) = \partial g(x) = 
    \co \, \{\nabla g_i(x) \mid 1 \leq i \leq m, \; g_i(x) = g(x) \}
  \end{equation*}
  where, for any function $f: \R^n \to \R$ and any $x \in \R^n$, $\widehat{\partial} f(x)$
  and $\partial f(x)$ denote respectively the {\em Fr\'echet subdifferential} and
  the {\em limiting/Mordukhovich subdifferential} of $f$ at $x$, see \cite{RW98} for their constructions.

  Now observe that $\prtb \in \R$ is a singular perturbation, that is, MFCQ is not
  satisfied at some point $x \in C_\prtb$, if and only if $g(x) = \prtb$
  and $0 \in \partial g( x)$.
  In other words, the singular diagonal perturbations of the inequality constraint
  set of \cref{pb:NLP} correspond to the {\em critical values}
  of $g$.

  Thus, when the functions $g_1,\dots,g_m$ are definable in the same o-minimal
  structure, so is $g$, and the finiteness of the singular diagonal perturbations
  stated in \cref{cor:diagonal-perturb} is equivalent to the finiteness of
  the critical values of the definable function $g$.
  Hence, with continuously differentiable functions, \cref{cor:diagonal-perturb}
  can be seen as a Sard-type theorem for definable functions, which was proved in full
  generality in \cite[Cor.~9]{BDLS07}.
  These arguments can also be extended when equality constraints with 
	linearly independant gradients are added to the constraint set,
  see \cref{rmk:regular-perturbations}~\labelcref{it:equality-constraints-i}.
\end{remark}

\subsection{A bound on the number  of singular perturbations for polynomial optimization}

We consider here constraint sets defined by real polynomial functions 
and we bound the number of singular values for the corresponding perturbed sets.

To tackle this problem, we evaluate the number of connected components of some adequate
real algebraic sets.
A key result regarding this evaluation is provided by Milnor-Thom's bound:
given any polynomial map $f: \R^p \to \R^q$, the number of connected components of the set
of zeros of $f$, $\{x \in \R^p \mid f(x) = 0\}$, is bounded  by
\begin{equation}
  \label{eq:Milnor-Thom}
  d \, (2d-1)^{p-1}
\end{equation}
where $d$ is the maximal degree of the polynomial functions $f_j$ for $j=1,\dots,q$,
see e.g., \cite{BLR91}.

\begin{theorem}
  \label{thm:Milnor-Thom}
  Let $g_1,\dots,g_m: \R^n \to \R$ be polynomial functions whose degree is bounded by $d$.
  Let $\{I,J\}$ be a partition of the set of indices $\{1,\dots,m\}$, possibly trivial,
  such that the Mangasarian-Fromovitz constraint qualification holds throughout
  $[g_j \leq 0,\, j \in J]$.
  Then, for the perturbed sets $[g_i \leq \prtb, g_j \leq 0, i \in I, j \in J]$
  with $\prtb$ ranging in $\R$, the number of singular perturbations is bounded by
  \[
    d \, (2d-1)^{n} \, (2d+1)^{m} .
  \]
\end{theorem}

\begin{proof}
  Denoting by $|J|$ the cardinality of $J$, we assume that $|J| \in \{0,\dots,m-1\}$
  (otherwise $I$ is empty and there is nothing to prove).
  For $\prtb \in \R$, let
  \[
    C_{I,\prtb} := [g_i \leq \prtb, g_j \leq 0, \, i \in I, \, j \in J] .
  \]
  If $\prtb$ is a singular perturbation, then there exists a point $x \in C_{I,\prtb}$
  for which $0 \in \co \, \{ \nabla g_i(x) \mid i \in I(x) \}$.
  So, there exists a subset of indices $K \subset I(x)$, which we fix, and positive scalars
  $\lambda_i > 0$, $i \in K$, such that $\sum_{i \in K} \lambda_i \, \nabla g_i(x) = 0$ and
  $\sum_{i \in K} \lambda_i = 1$. 
  Furthermore, $K \not \subset J$ since MFCQ holds throughout $[g_j \leq 0, j \in J]$.
  Let $L$ be the set of indices equal to $I(x)$.
  Thus, the sets $K$ and $L$ being fixed, the tuple $(x,\lambda,\prtb) \in \R^n \times \R^K \times \R$
  is solution of the polynomial system
  \begin{equation}
    \label{eq:singular-prtb-system}
    \left\{
    \begin{aligned}
      & \sum_{i \in K} \lambda_i \, \nabla g_i(x) = 0 , \\
      & \sum_{i \in K} \lambda_i = 1 , \\
      & g_j(x) = \prtb , \quad j \in L \cap I , \\
      & g_j(x) = 0 , \quad j \in L \cap J , \\
    \end{aligned}
    \right.
  \end{equation}
  and satisfies the following additional constraints: $\lambda \in \R_{++}^K$,
  $g_\ell(x) < \prtb$ for all $\ell \in I \setminus L$, and $g_\ell(x) < 0$
  for all $\ell \in J \setminus L$.
	
	The first step of the proof is to show that the number of singular perturbations is
   bounded above by the number of connected components of the set of solutions of
  \labelcref{eq:singular-prtb-system} for all possible choices of $K$ and $L$. 
	This is done by constructing an injection from the set of singular perturbations to
  these connected components.

  Fix a singular value $\prtb$ and
  choose a subset $L \subset \{1,\dots,m\}$ with maximal cardinality among all the sets
  of active constraints $I(x)$ such that MFCQ is not satisfied at $x \in C_{I,\prtb}$.
  Then pick a subset $K \subset L$ with minimal cardinality among all the subsets
  $K' \subset L$ such that the system~\labelcref{eq:singular-prtb-system} with $K$ replaced by $K'$
  has a solution $(x,\lambda,\prtb) \in \R^n \times \R^{K'} \times \R$ with
  $x \in C_{I,\prtb}$ and $\lambda \in \R_+^{K'}$.
  Let $(\bar{x},\bar{\lambda},\prtb)$ be such a solution for the particular choice of
  $K$ and $L$.
  Note that $K \not \subset J$ since $[g_j \leq 0, j \in J]$ is regular.
  Also note that $\bar{\lambda} \in \R_{++}^K$ by minimality of $|K|$,
  and that $g_\ell(\bar{x}) < \prtb$ for all $\ell \in I \setminus L$, and $g_\ell(\bar{x}) < 0$
  for all $\ell \in J \setminus L$ by maximality of $|L|$.

  Let $Q \subset \R^n \times \R^K \times \R$ be the connected component of
  the set of solutions of~\labelcref{eq:singular-prtb-system} corresponding to $K$
  and $L$, containing the tuple $(\bar{x},\bar{\lambda},\prtb)$.
  We next prove that
  \[
    Q \subset S(\prtb,K,L) := 
    \{x \in C_{I,\prtb} \mid I(x) = L\} \times \R_{++}^K \times \{ \prtb \} .
  \]
  Toward a contradiction, assume that the above inclusion does not hold.
  There exists therefore a continuous path $(x(\cdot), \lambda(\cdot), \alpha(\cdot))$ from
  $[0,1]$ to $Q$ such that
  \[
    \begin{cases}
      (x(0), \lambda(0), \alpha(0)) = (\bar{x},\bar{\lambda},\prtb) , \\
      (x(1), \lambda(1), \alpha(1)) \notin S(\prtb,K,L) .
    \end{cases}
  \]
  Let $t = \sup \{s \in [0,1] \mid (x(s), \lambda(s), \alpha(s)) \in S(\prtb,K,L) \}$.
  By continuity we have $\alpha(t) = \prtb$, $x(t) \in C_{I,\prtb}$ and
  $\lambda(t) \in \R_+^K$.
  If either $I(x(t)) \neq L$ or $\lambda(t) \notin \R_{++}^K$, then there is a contradiction
  with the maximality of $|L|$ (since we already have $L \subset I(x(t))$) or
  with the minimality of $|K|$.
  Hence, we have $I(x(t)) = L$ and $\lambda(t) \in \R_{++}^K$. 
	Since in addition $x(t) \in C_{I,\prtb}$ and $\alpha(t) = \alpha$, we have
  $(x(t), \lambda(t), \alpha(t)) \in S(\prtb,K,L)$. 
	Finally, we have $t < 1$ since $(x(1), \lambda(1), \alpha(1)) \not\in S(\prtb,K,L)$.
	Using the continuity of the path and \labelcref{eq:singular-prtb-system}, there exists
  $\varepsilon > 0$ such that for all $s \in \left[ t, t+\varepsilon \right)$,
  we have $x(s) \in C_{I,\prtb(s)}$ with $I(x(s)) = L$ and $\lambda(s) \in \R_{++}^K$.
  This implies that $\alpha(s)$ is a singular perturbation for all
  $s \in \left[ t, t+\varepsilon \right)$.
  Combining the continuity of $\alpha(\cdot)$ and \cref{cor:diagonal-perturb},
  $\alpha(\cdot)$ is constant on $\left[ t, t+\varepsilon \right)$. 
	Hence, we have $(x(s), \lambda(s), \alpha(s)) \in S(\prtb,K,L)$ for all
  $s \in \left[ t, t+\varepsilon \right)$. 
	From the definition of $t$, we obtain $t \geq t + \varepsilon$ which is contradictory
  since $\varepsilon > 0$.

  Thus, for every singular perturbation $\prtb$, there exist subsets
  $K \subset L \subset \{1,\dots,m\}$ with $K \cap I \neq \emptyset$ such that
  the set of solutions of the polynomial system~\labelcref{eq:singular-prtb-system}
  with this choice of $K$ and $L$ has at least one connected component included
  in $\R^n \times \R^K \times \{\prtb\}$.
	Hence the mapping sending every singular perturbation to this connected component
  is injective.
	So we have just proved that the number of singular perturbations is upper bounded 
	by the number of connected components of the set of solutions of
  \labelcref{eq:singular-prtb-system} for all possible choices of $K$ and $L$. 
  We can then deduce from Milnor-Thom's bound \labelcref{eq:Milnor-Thom} an upper bound
  for the number of singular perturbations $\prtb$ by summation over all possible
  choices of $K$ and $L$. 

  Denote $p = |I| \in \{1,\dots,m\}$.
  In the computation below we denote by $\ell_1, \ell_2$ the cardinality of $L \cap I$ and
  $L \cap J$, respectively, and by $k_1, k_2$ the cardinality of $K \cap I$ and
  $K \cap J$, respectively.
  Since the system \labelcref{eq:singular-prtb-system} has degree $d$ and
  $n + k_1 + k_2 + 1$ variables, the number of singular perturbation is bounded by
  \begin{align*}
    & \sum_{\substack{ 1 \leq \ell_1 \leq p \\ 0 \leq \ell_2 \leq m-p}}
    \mybinom{p}{\ell_1} \mybinom{m-p}{\ell_2}
    \sum_{\substack{ 1 \leq k_1 \leq \ell_1 \\ 0 \leq k_2 \leq \ell_2}}
    \mybinom{\ell_1}{k_1} \mybinom{\ell_2}{k_2} \,
    d \, (2d-1)^{n+k_1+k_2} \\
    & \qquad = d \, (2d-1)^n \, (2d+1)^{m-p} \, \big( (2d+1)^p - 2^p \big) , \\
    & \qquad = d \, (2d-1)^n \, (2d+1)^m \, \left( 1 - \left( \frac{2}{2d+1} \right)^p \right) .
  \end{align*}
  To conclude, observe that
  \begin{equation}
    \label{eq:IJ}
    \frac{1}{3} \leq 1- \left( \frac{2}{2d+1} \right)^p \leq 1
  \end{equation}
  for all $d \geq 1$ and $1 \leq p \leq m$.
\end{proof}

\begin{remark}
  \label{rmk:equality-constraints-v}
  \begin{enumerate}[label=(\alph*),leftmargin=0pt,itemindent=0.25in,labelsep=3pt]
    \item As attested in a forthcoming example the choice of a partition $(I,J)$ has a very
      marginal impact on the global bound which we have neglected in our main estimate
      \labelcref{eq:Milnor-Thom}\footnote{Our proof shows that its evolves within
        the interval~$[1/3,1]$, see \labelcref{eq:IJ}.}.
    \item Let $h_1, \dots, h_r: \R^n \to \R$ be polynomial functions with maximal degree $d$
      such that the set $[h_1 = 0, \dots, h_r = 0]$ satisfies MFCQ
      (i.e, the first regularity assumption in \cref{def:MFCQ}). 
      Then, with a minor adaptation of the above proof, we can show that for the perturbed sets
      $[g_i \leq \prtb, g_j \leq 0, i \in I, j \in J] \cap [h_1 = 0, \dots, h_r = 0]$ the number of
      singular perturbations $\prtb \in \R$ is bounded by
      \[
        d (2d-1)^{n+r} (2d+1)^m .
      \]
      Indeed, if $\prtb$ is singular, then there exists a tuple
      $(x,\lambda,\kappa,\prtb) \in \R^n \times \R_{++}^K \times \R^r \times \R$ that is
      solution of a polynomial system similar to \labelcref{eq:singular-prtb-system}
      with the following changes:
      \begin{itemize}
        \item add the $r$ equality constraints $h_j(x) = 0$, $j = 1, \dots, r$;
        \item replace the right-hand side of the first equality by
          the linear combination\\
          $\sum_{j=1}^r \kappa_j \, \nabla h_j(x)$.
      \end{itemize}
      The rest of the proof follows the exact same lines, with a trivial adaptation
      of the notation.
      In particular, we do not need to take into account the values of the coefficients $\kappa_j$,
      $j = 1, \dots, r$, contrary to the $\lambda_i$, $i \in K$.
      Using the same notation as in the proof, this new system has degree $d$ and $n+k_1+k_2+r+1$
      variables.
      Whence the bound.
  \end{enumerate}
\end{remark}

Milnor-Thom's bound \labelcref{eq:Milnor-Thom} and a fortiori the bound in
\cref{thm:Milnor-Thom} is not sharp, but one may ask whether it is of the right
order of magnitude.
The following examples show that this is indeed the case, at least regarding the dependence 
with respect to the degree $d$ of the polynomials and the dimension $n$ of the base space.
They also illustrate the absence of sensitivity of our bound with respect to the choice of the partition $(I,J)$.

Indeed, the examples show that even if all but one constraints define a regular set,
the number of singular perturbations generated by the last constraint is of the right order.
In the first example, which is thoroughly explained, the degree is fixed to $d=2$,
and the number of singular perturbations is shown to be exponential with respect to $n$.
In the second example, the number of singular diagonal perturbations is shown to be
highly dependent on the degree $d$.

\begin{example}
  \label{ex:2n-singular}
  Here, we construct an inequality constraint  set in $\R^n$ defined by $n+1$ polynomial
  functions of degree $2$, $n$ of which are convex.
  The number of singular perturbations corresponding to a variation of the unique
  nonconvex constraint is  $3^n-1$.

  Let $a \in \R^n$ be a point in $(-1,1)^n$.
  Then, for $\prtb \in \R$, define the constraint set $C_{0,\prtb}$
  as the set of points $x \in \R^n$ such that
  \[
    \left\{
      \begin{aligned}
        g_0(x) & = 4n - \sum_{i=1}^n (x_i-a_i)^2 \leq \prtb , \\
        g_i(x) & = (x_i)^2 \leq 1 , \quad i = 1,\dots,n .
      \end{aligned}
    \right.
  \]
  For $\prtb < 4n$, the first inequality defines the complement in $\R^n$ of the open
  ball centered at point $a$ with radius $\sqrt{4n-\prtb}$, denoted by $B(a,\sqrt{4n-\prtb})$.
  As for the $n$ last inequalities, they are convex and define the hypercube $[-1,1]^n$.
  Observe that for $\prtb \leq 0$, $C_{0,\prtb}$ is empty since $[-1,1]^n$ is strictly included in
  $B(a,\sqrt{4n-\prtb})$, whereas for $\prtb \geq 4n$, $C_{0,\prtb} = [-1,1]^n$.
  
  We next show that a perturbation $\prtb$ is singular whenever
  a face of $[-1,1]^n$ and the ball $B(a,\sqrt{4n-\prtb})$ are tangent. 
  First note that the constraint sets $[g_0 \leq \prtb]$ and $[g_1 \leq 1,\cdots,g_n \leq 1]$
  both satisfies MFCQ.
  Hence, the constraint qualification for $C_{0,\prtb}$ may fail only at points where the
  constraint $g_0$ and at least one of the constraints $g_i$ with $i \in \{1,\cdots,n\}$ are active,
  that is, at intersection points between the boundary of the hypercube $[-1,1]^n$ and
  the boundary of the ball $B(a,\sqrt{4n-\prtb})$.

  Let $z$ be such a point for a given $\prtb$.
  There exists a nonempty subset of indices $I$ and integers $v_i \in \{\pm 1\}$ for $i \in I$
  such that $z_i = v_i$ for all $i \in I$ and $|z_j| < 1$ for all $j \notin I$.
  Then MFCQ is not satisfied at $z$ if and only if the convex hull of the gradients $\nabla g_0(z)$
  and $\nabla g_i(z)$ with $i \in I$ contains $0$, and since $[-1,1]^n$ is qualified, this is
  equivalent to $-\nabla g_0(z)$ being in the convex cone generated by the gradients
  $\nabla g_i(z)$, $i \in I$.
  Since $-\nabla g_0(z) = 2(z-a)$ and $\nabla g_i(z) = 2 v_i e_i$ for every $i \in I$, where $e_i$
  denotes the $i$th coordinate vector of $\R^n$, the latter condition holds if and only if  
  $z_j = a_j$ for all $j \notin I$, i.e., if and only if $z$ is the orthogonal projection of $a$
  on the face $F = \{ x \in \R^n \mid x_i = v_i, \; i \in I, \; |x_j| \leq 1, \; j \notin I\}$.
  In other words, MFCQ is not satisfied at point $z$ if and only if a face of the hypercube $[-1,1]^n$
  and the ball $B(a,\sqrt{4n-\prtb})$ are tangent at $z$.

  Now, given $k \in \{0,\ldots,n-1\}$ there are  ${n \choose k} 2^{n-k}$ faces of dimension $k$
  in the cube  $[-1,1]^n$.
  Thus, by choosing adequately $a$ in $(-1,1)^n$ so that, for all $\prtb$, $B(a,\sqrt{4n-\prtb})$
  is tangent to a unique face of $[-1,1]^n$ at most, we deduce that the total number of singular
  perturbations is $\sum_{k=0}^{n-1} {n \choose k} 2^{n-k} = 3^n-1$.
  \Cref{fig:2n-singular} shows a representation of $C_{0,\prtb}$ in $\R^2$ for each singular
  value $\prtb$.

  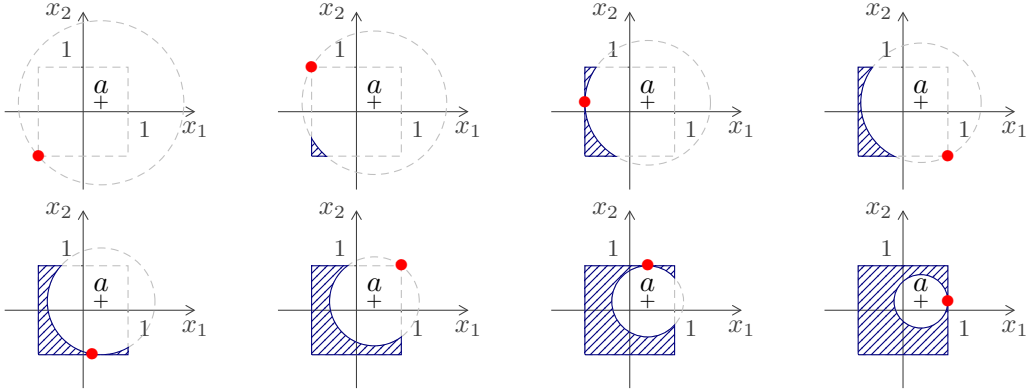
\begin{figure}[bht]
    \centering
    \def\scale{0.59}
    \def\ctr{2/5,1/5}
    \def\radlim{1.85}
    \def\rect{(-1, -1) rectangle (1, 1)}
    \def\rada{1.844}
    \def\circa{(\ctr) circle (\rada)}
    \def\radb{1.612}
    \def\circb{(\ctr) circle (\radb)}
    \def\radc{1.4}
    \def\circc{(\ctr) circle (\radc)}
    \def\radd{1.342}
    \def\circd{(\ctr) circle (\radd)}
    \def\rade{1.2}
    \def\circe{(\ctr) circle (\rade)}
    \def\radf{1.0}
    \def\circf{(\ctr) circle (\radf)}
    \def\radg{0.8}
    \def\circg{(\ctr) circle (\radg)}
    \def\radh{0.6}
    \def\circh{(\ctr) circle (\radh)}
    \def\axes{
      \draw[->, >= angle 60,black!75] (-1.75,0) -- (2.5,0);
      \draw[black!75] (2.5,0) node[below] {\small $x_1$};
      \draw[black!75] (1,0) -- (1,-0.05);
      \draw[black!75] (1,0) node[below right] {\footnotesize $1$};
      \draw[->, >= angle 60,black!75] (0,-1.75) -- (0,2.25);
      \draw[black!75] (0,2.25) node[left] {\small $x_2$};
      \draw[black!75] (-0.05,1) -- (0,1);
      \draw[black!75] (0,1) node[above left] {\footnotesize $1$};
    }

    \hspace{\stretch{2}}
    \begin{tikzpicture}[scale=\scale]
      \axes
      \draw (\ctr) node {\tiny $+$};
      \draw (\ctr) node [above] {$a$};
      \draw[gray!50,densely dashed] \circa;
      \draw[gray!50,densely dashed] \rect;
      \draw[red] (-1, -1) node {$\bullet$};
    \end{tikzpicture}
    \hspace{\stretch{1}}
    \begin{tikzpicture}[scale=\scale]
      \axes
      \draw (\ctr) node {\tiny $+$};
      \draw (\ctr) node [above] {$a$};
      \draw[gray!50,densely dashed] \circb;
      \draw[gray!50,densely dashed] \rect;
      \begin{scope}[even odd rule]
        \clip \circb (\ctr) circle (\radlim);
        \draw[myblue] \rect;
        \fill[pattern=north east lines,pattern color=myblue] \rect;
      \end{scope}
      \begin{scope}
        \clip \rect;
        \draw[myblue] \circb;
      \end{scope}
      \draw[red] (-1, 1) node {$\bullet$};
    \end{tikzpicture}
    \hspace{\stretch{1}}
    \begin{tikzpicture}[scale=\scale]
      \axes
      \draw (\ctr) node {\tiny $+$};
      \draw (\ctr) node [above] {$a$};
      \draw[gray!50,densely dashed] \circc;
      \draw[gray!50,densely dashed] \rect;
      \begin{scope}[even odd rule]
        \clip \circc (\ctr) circle (\radlim);
        \draw[myblue] \rect;
        \fill[pattern=north east lines,pattern color=myblue] \rect;
      \end{scope}
      \begin{scope}
        \clip \rect;
        \draw[myblue] \circc;
      \end{scope}
      \draw[red] (-1, 1/5) node {$\bullet$};
    \end{tikzpicture}
    \hspace{\stretch{1}}
    \begin{tikzpicture}[scale=\scale]
      \axes
      \draw (\ctr) node {\tiny $+$};
      \draw (\ctr) node [above] {$a$};
      \draw[gray!50,densely dashed] \circd;
      \draw[gray!50,densely dashed] \rect;
      \begin{scope}[even odd rule]
        \clip \circd (\ctr) circle (\radlim);
        \draw[myblue] \rect;
        \fill[pattern=north east lines,pattern color=myblue] \rect;
      \end{scope}
      \begin{scope}
        \clip \rect;
        \draw[myblue] \circd;
      \end{scope}
      \draw[red] (1, -1) node {$\bullet$};
    \end{tikzpicture}
    \hspace{\stretch{2}}

    \hspace{\stretch{2}}
    \begin{tikzpicture}[scale=\scale]
      \axes
      \draw (\ctr) node {\tiny $+$};
      \draw (\ctr) node [above] {$a$};
      \draw[gray!50,densely dashed] \circe;
      \draw[gray!50,densely dashed] \rect;
      \begin{scope}[even odd rule]
        \clip \circe (\ctr) circle (\radlim);
        \draw[myblue] \rect;
        \fill[pattern=north east lines,pattern color=myblue] \rect;
      \end{scope}
      \begin{scope}
        \clip \rect;
        \draw[myblue] \circe;
      \end{scope}
      \draw[red] (1/5, -1) node {$\bullet$};
    \end{tikzpicture}
    \hspace{\stretch{1}}
    \begin{tikzpicture}[scale=\scale]
      \axes
      \draw (\ctr) node {\tiny $+$};
      \draw (\ctr) node [above] {$a$};
      \draw[gray!50,densely dashed] \circf;
      \draw[gray!50,densely dashed] \rect;
      \begin{scope}[even odd rule]
        \clip \circf (\ctr) circle (\radlim);
        \draw[myblue] \rect;
        \fill[pattern=north east lines,pattern color=myblue] \rect;
      \end{scope}
      \begin{scope}
        \clip \rect;
        \draw[myblue] \circf;
      \end{scope}
      \draw[red] (1, 1) node {$\bullet$};
    \end{tikzpicture}
    \hspace{\stretch{1}}
    \begin{tikzpicture}[scale=\scale]
      \axes
      \draw (\ctr) node {\tiny $+$};
      \draw (\ctr) node [above] {$a$};
      \draw[gray!50,densely dashed] \circg;
      \draw[gray!50,densely dashed] \rect;
      \begin{scope}[even odd rule]
        \clip \circg (\ctr) circle (\radlim);
        \draw[myblue] \rect;
        \fill[pattern=north east lines,pattern color=myblue] \rect;
      \end{scope}
      \begin{scope}
        \clip \rect;
        \draw[myblue] \circg;
      \end{scope}
      \draw[red] (2/5, 1) node {$\bullet$};
    \end{tikzpicture}
    \hspace{\stretch{1}}
    \begin{tikzpicture}[scale=\scale]
      \axes
      \draw (\ctr) node {\tiny $+$};
      \draw (\ctr) node [above] {$a$};
      \draw[gray!50,densely dashed] \circh;
      \draw[gray!50,densely dashed] \rect;
      \begin{scope}[even odd rule]
        \clip \circh (\ctr) circle (\radlim);
        \draw[myblue] \rect;
        \fill[pattern=north east lines,pattern color=myblue] \rect;
      \end{scope}
      \begin{scope}
        \clip \rect;
        \draw[myblue] \circh;
      \end{scope}
      \draw[red] (1, 1/5) node {$\bullet$};
    \end{tikzpicture}
    \hspace{\stretch{2}}
    \caption{Singular perturbations of a constraint set (hatched area)
      defined by degree~2 polynomials.}
    \label{fig:2n-singular}
  \end{figure}
\end{example}

\begin{remark}
  The dependence of the number of singular values in the previous exam\-ple, $3^n-1$,
  with respect to $m$ and $n$ does not appear clearly since $m = n+1$.
  In this regard, the gap between this number and the bound predicted by
  \cref{thm:Milnor-Thom}, $2 \times 3^n \times 5^{n+1} = 10 \times 15^n$, questions
  the relevance of the exponential term in $m$ appearing in \cref{thm:Milnor-Thom}.
  In order to better understand this dependence, we could think of an example
  similar to \cref{ex:2n-singular} where the hypercube would be replaced by a polytope
  with $m$ facets, hence defined by $m$ linear constraints (instead of $2n$ in
  \cref{ex:2n-singular}).
  However, the maximum number of vertices of such a polytope, given by the upper bound
  theorem \cite{McM70}, is asymptotically equal to $O(m^{\lfloor n/2 \rfloor})$
  (see~\cite{Sei95}).
  Hence, such an example could not have a number of singular perturbations exponential
  with respect to $m$.
  It then remains an open question to understand the dependence of the maximum number
  of singular values with respect to $m$, $n$.
\end{remark}

\begin{example}
  We build an example in $\R^n$
  with $n+1$ polynomial constraints, degree $2d$, and
  we show that the perturbation of a unique constraint generates at least
  $d^n$ singular values.

  For any even integer $d$, let us consider the polynomial
  $Q_d = \prod_{k=1}^d (X^2 - k^2)$.
  Let $H$ be the set of points $x \in \R^n$ such that $g_i(x) = Q_d(x_i) \leq 0$,
  $i = 1,\dots,n$.
  The set $H$ is a disjoint union of $d^n$ boxes.
  More precisely, since $d$ is even, $Q_d$ is nonpositive on the intervals
  $[2k-1, 2k]$, $k=1,\dots,d/2$, and on their symmetrical images with respect to $0$.
  Then $H$ is the (disjoint) union of the $d^n$ boxes
  \begin{equation}
    \label{eq:boxes}
    H(v,k) := \prod_{i=1}^n v_i \, [2k_i-1, 2k_i] , \quad
    v \in \{ \pm 1 \}^n , \quad k \in \left\{1, \dots, d/2 \right\}^n .
  \end{equation}
  Note that all the boxes \labelcref{eq:boxes} are included in $[-d,d]^n$.
  Let $a$ be some point in $(-d,d)^n$ and for $\prtb \in \R$, define $C_{0,\prtb}$ as
  the set of points contained in $H$ and that satisfy in addition
  $g_0(x) = 4 n d^2 - \|x-a\|^2 \, \leq \, \prtb$.
  \Cref{fig:dn-singular} displays $C_{0,\prtb}$ for $d = 4$.

  \begin{figure}[htb]
    \centering
    \begin{tikzpicture}[scale=0.66]
      \draw[->, >= angle 60,black!75] (-4.5,0) -- (4.5,0);
      \draw[black!75] (4.5,0) node[below] {\small $x_1$};
      \foreach \x in {-4,...,4}
      \draw[black!75] (\x,0) -- (\x,-0.05);
      \draw[black!75] (1,0) node[below] {\footnotesize $1$};
      \draw[->, >= angle 60,black!75] (0,-4.25) -- (0,4.5);
      \draw[black!75] (0,4.5) node[left] {\small $x_2$};
      \foreach \y in {-4,...,4}
      \draw[black!75] (0,\y) -- (-0.05,\y);
      \draw[black!75] (0,1) node[left] {\footnotesize $1$};
      \foreach \x in {-4,-2,1,3} \foreach \y in {-4,-2,1,3}
      \draw[gray!50,densely dashed] (\x,\y) rectangle (\x+1,\y+1);
      \draw (3/5,2/5) node {\footnotesize $+$};
      \draw (3/5,2/5) node [above] {$a$};
      \draw[gray!50,densely dashed] (3/5,2/5) circle (3.053);
      \begin{scope}[even odd rule]
        \clip (3/5,2/5) circle (3.053) (-4,-4) rectangle (4,4);
        \foreach \x in {-4,-2,1,3} \foreach \y in {-4,-2,1,3}
        {
          \fill[pattern=north east lines,pattern color=myblue] (\x,\y) rectangle (\x+1,\y+1);
          \draw[myblue] (\x,\y) rectangle (\x+1,\y+1);
        }
      \end{scope}
      \begin{scope}
        \clip (-2,-2) rectangle (-1,-1);
        \draw[myblue] (3/5,2/5) circle (3.053);
      \end{scope}
      \begin{scope}
        \clip (3,-2) rectangle (4,-1);
        \draw[myblue] (3/5,2/5) circle (3.053);
      \end{scope}
      \begin{scope}
        \clip (3,1) rectangle (4,2);
        \draw[myblue] (3/5,2/5) circle (3.053);
      \end{scope}
      \begin{scope}
        \clip (1,3) rectangle (2,4);
        \draw[myblue] (3/5,2/5) circle (3.053);
      \end{scope}
      \begin{scope}
        \clip (-2,3) rectangle (-1,4);
        \draw[myblue] (3/5,2/5) circle (3.053);
      \end{scope}
      \foreach \x in {-4,-2,2,4} \foreach \y in {-4,-2,2,4}
      \draw[red] (\x,\y) node {\footnotesize $\bullet$};
    \end{tikzpicture}
    \caption{Constraint set (hatched area) defined by $n+1$ polynomials of degree
      at most $2d$ with $d^n$ singular perturbations ($n=2$, $d=4$).}
    \label{fig:dn-singular}
  \end{figure}
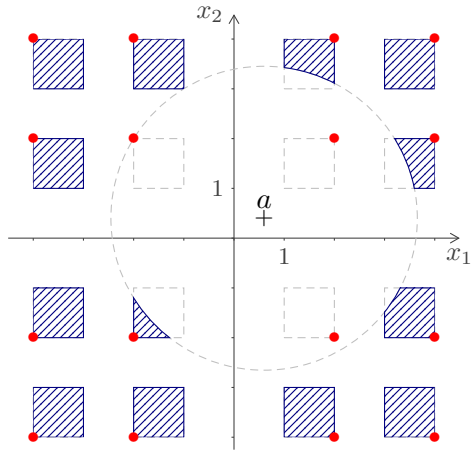

  Let us follow the  arguments of \cref{ex:2n-singular}. Observe that
  for each vertex $v \in \{ \pm 1 \}^n$ and for each tuple of indices
  $k \in \{1,\dots,d/2\}^n$, there exists a unique perturbation $\prtb \in (0, 4 n d^2)$
  such that MFCQ does not hold at point $(2 k_i \, v_i)_{1 \leq i \leq n}$, that is,
  when the box $H(v,k)$ defined in~\labelcref{eq:boxes} and the sphere centered at $a$
  with radius $\sqrt{4 n d^2 - \prtb}$ have a unique contact point
  (see \cref{fig:dn-singular}).
  Finally, by choosing adequately $a$, it is possible to show that all the $d^n$ singular
  perturbations mentioned above are distinct.
\end{example}

\section{Applications to optimization algorithms}
\label{sec:applications}

We illustrate here the results of \Cref{sec:finite-singularities} through
some classical algorithms for nonlinear optimization.
Our approach consists in embedding the original problem within some one-parameter family
of optimization problems:
\begin{equation}
  \label[prb]{pb:parametric}
  \tag{$\PNLP$}
  \begin{aligned}
    \text{minimize} & \enspace f(x) \\
    \text{subject to} & \enspace g_1(x) \leq \prtb, \dots, g_m(x) \leq \prtb,
  \end{aligned}
\end{equation}
where $f, g_1, \dots, g_m:\R^m\to \R$ are differentiable definable functions.

A first obvious but important consequence is that any algorithm which is ope\-ratio\-nal
under the standard qualification condition can be applied to \cref{pb:parametric}
except perhaps for a finite number of parameters $\prtb$.
In view of the fact that $\val \PNLP$ tends to $\val (\mathcal{P}_0)$ (see
\cref{lem:continuity-value}), it provides a natural way of approximating $(\mathcal{P}_0)$.
This can be illustrated in a straightforward manner with many types of algorithms,
see e.g., \cite{NW06,GW12,Ber16}, see also \cite{homotopy} for continuation techniques in optimization.
Estimating the complexity of such an approach is a matter for future research\footnote{It is likely
to be connected to the results from \cite{Fra90} and \cite{FQ12}.}.
In this spirit of a direct approximation, we provide an illustration
involving SDP relaxations on the KKT ideal which improves a series of results
of \cite{DNS06,DNP07,ABM14}.

Another family of applications considered below is provided by infeasible SQP methods which 
often require strong qualification conditions assumptions.

\subsection{Infeasible Sequential Quadratic Programming}

We consider the {\em Extended Sequential
Quadratic Method}, \ESQM, proposed by Auslender~\cite{Aus13} and based on
a $\ell^\infty$ penalty function.
Other methods could be treated as, for instance, Flechter's  S$\ell^1$QP \cite{Fle85}. We make the following very basic assumptions:

\begin{assumption}
  \label{asm:esqm}
  \leavevmode
  \begin{enumerate}
    \item {\em (Regularity).} The functions $f, g_1, \dots, g_m: \R^n \to \R$ are
      $\mathcal{C}^2$ with Lipschitz continuous gradients.
      We denote by $L, L_1,\dots,L_m > 0$ their Lipschitz constants, respectively.
      \label{it:regularity}
    \item {\em (Compactness).} The constraint sets
      $C_\prtb = [g_1 \leq \prtb,\dots,g_m \leq \prtb]$ are compact and nonempty
      for all $\prtb \geq 0$.
      \label{it:compact}
    \item {\em (Boundedness).} $\inf_{x \in \R^n} f(x) > -\infty$.
      \label{it:boundedness}
    \end{enumerate}
\end{assumption}

The general SQP method we consider, \ESQM, is described below.
The strength of the following general convergence theorem is to rely merely on
semi-alge\-brai\-city/de\-fina\-bility and boun\-ded\-ness assumptions.
In particular, it does not require any qualification assumptions whatsoever.
Another distinctive feature of this result is to allow to treat all at once many issues
such as nonconvexity, continuum of stationary points, infeasibility, nonlinear constraints or
oscillations (see \cite{BP16} for more on the key issues).

\begin{algorithm}[htb]
  \caption{-- Extended Sequential Quadratic Method \cite{Aus13,BP16}}
  \label{algo:esqm}
  \begin{algorithmic}
    \STATE {\bfseries Step 1}:
    Choose $x_0 \in C_\prtb$, $\beta_0 > 0$, $\delta > 0$, $\lambda \geq L$ and
    $\lambda' \geq \max_{i} L_i$, and set $k \leftarrow 0$.
    \STATE \label{it:min-pb} {\bfseries Step 2}:
    Compute $x_{k+1}$ solution (along with some $s \in \R$) of
    \vskip-1\baselineskip
    \begin{align*}
      \operatorname*{minimize}_{s \in \R , \: y \in \R^n} \enspace & f(x_k) + \<\nabla f(x_k),y-x_k> +
      \beta_k s + \frac{\lambda + \beta_k \lambda'}{2} \|y-x_k\|^2 \\
      \text{s.t.} \enspace & g_i(x_k) + \<\nabla g_i(x_k),y-x_k> \leq \prtb + s
      \enspace , \quad i = 1,\dots,m, \\
      & s \geq 0.
    \end{align*}
    \vskip-0.5\baselineskip
    \STATE {\bfseries Step 3}: If \enspace
    $g_i(x_k) + \<\nabla g_i(x_k),x_{k+1}-x_k> \leq \prtb$, \enspace $i = 1,\dots,m$,
    \enspace then \enspace $\beta_{k+1} \leftarrow \beta_k$. \\
    \newlength{\myl} \settowidth{\myl}{{\bfseries Step 3}:\ }
    \hspace{\myl}Else \enspace $\beta_{k+1} \leftarrow \beta_k + \delta$.
    \STATE {\bfseries Step 4}: $k \leftarrow k+1$, go to step 2.
  \end{algorithmic}
\end{algorithm}

\begin{theorem}[Large penalty parameters yield convergence of ESQM]
  \label{thm:esqm}
  Assume that \cref{it:regularity,it:compact,it:boundedness} hold (smoothness,
  compactness, boundedness).
  For all parameters $\prtb \geq 0$, except for a finite number of them, 
  there exists a number $\beta(\prtb)\geq 0$ such that {\rm \ESQM}
  initialized with any $\beta_0 \geq \beta(\prtb)$ generates
  a sequence $(x_k)_{k \in \N}$ that converges to some KKT point of \cref{pb:parametric}.
\end{theorem}

\begin{proof}
  Following \cref{cor:diagonal-perturb}, there exist a finite family of parameters
  $\mathcal{A} \subset \R_+$ such that for all $\prtb \notin \mathcal{A}$, MFCQ holds
  throughout $C_\prtb$.
  Let us fix a parameter $\prtb \in \R_+ \setminus \mathcal{A}$.
  Then there exists a positive real number $\varepsilon$ such that
  $[\prtb,\prtb+\varepsilon] \subset \R_{+} \setminus \mathcal{A}$.
  This implies that for every $x \in C_{\prtb+\varepsilon}$, if $g_j(x) \geq \prtb$
  for some index $j$, then there exist $y \in \R^n$ such that $\<y, \nabla g_i(x) > < 0$
  for all indices $i$ such that $g_i(x) = \max_{1 \leq \ell \leq m} g_\ell(x)$.
  This follows from the fact that $x \in C_{\prtb'}$ where
  $\prtb' = \max_{1 \leq i \leq m} g_i(x)$ is such that
  $\prtb \leq \prtb' \leq \prtb + \varepsilon$, so that MFCQ holds at $x \in C_{\prtb'}$.

  Let $f_{\min} = \inf_{x \in \R^n} f(x)$ and $g_0$ be the constant function equal to $\prtb$.
  Set $d_k = x_{k+1} - x_k$.
  For every $k \in \N$, we have
  \begin{align*}
    & \frac{1}{\beta_{k+1}} (f(x_{k+1}) - f_{\min}) + \max_{0 \leq i \leq m} g_i(x_{k+1}) \\
    & \hspace{6em} \leq \frac{1}{\beta_k} (f(x_{k+1}) - f_{\min}) +
    \max_{0 \leq i \leq m} g_i(x_{k+1}) , \\
    & \hspace{6em} \leq \frac{1}{\beta_k} (f(x_k) + \< \nabla f(x_k), d_k > - f_{\min}) \\
    & \hspace{12em} + \max_{0 \leq i \leq m} \big( g_i(x_k) + \< \nabla g_i(x_k), d_k > \big) +
    \frac{\lambda + \beta_k \lambda'}{2 \beta_k} \|d_k\|^2 , \\
    & \hspace{6em} \leq \frac{1}{\beta_k} (f(x_k) - f_{\min}) + \max_{0 \leq i \leq m} g_i(x_k) ,
  \end{align*}
  where the second inequality comes from the Lipschitz continuity of the gradients of
  the functions invovled, and the third inequality follows from the minimization problem
  in \cref{it:min-pb} of~\ESQM.

  Now choose $\beta_0 \geq (f(x_0)-f_{\min}) / \varepsilon$.
  By a trivial induction, we deduce that, for every integer $k \in \N$,
  \[
    \max_{0 \leq i \leq m} g_i(x_k) 
    \leq \frac{1}{\beta_0} (f(x_0) - f_{\min}) + \max_{0 \leq i \leq m} g_i(x_0) 
    \leq \prtb + \varepsilon .
  \]
  Hence, all the points $x_k$ generated by \ESQM with the latter choice of $\beta_0$
  lie in $C_{\prtb+\varepsilon}$, and so satisfy the following qualification condition
  (an essential ingredient in \cite{BP16}): if $g_j(x) \geq \prtb$ for some index $j$ and
  some $x$ in $\R^n$, then there exists $y \in \R^n$ such that $\<y, \nabla g_i(x) > < 0$
  for all indices $i$ such that $g_i(x) = \max_{1 \leq \ell \leq m} g_\ell(x)$. 
  
  The fact that any cluster point of $(x_k)_{k \in \N}$ is a KKT point of
  \cref{pb:parametric} readily follows from \cite[Th.~2]{BP16}
  (see also \cite[Th.3.1]{Aus13}).
  The convergence of $(x_k)_{k \in \N}$, follows
  then from \cite[Th.~3]{BP16} and  the definability assumptions.
\end{proof}

\begin{remark}[Stabilization of penalty parameters]
  \begin{enumerate}[label=(\alph*),leftmargin=0pt,itemindent=0.25in,labelsep=3pt]
    \item For a fixed $\prtb$, the sequence of penalty parameters $\beta_k$ is constant
      after a finite number of iterations.
      This was already an essential result in \cite{Aus13} which still holds here.
    \item As in \cite{BP16}, rates of convergence are available when the data are in addition
      real semi-algebraic.
  \end{enumerate}
\end{remark}

\subsection{Exact relaxation in polynomial programming}

A standard approach for solving \cref{pb:parametric} when data are polynomial
relies on hierarchies of semidefinite programming, see \cite{Las01,Las10}.
It is known that, generically, these hierarchies are exact, meaning that they converge
in a finite number of steps (see \cite{Nie14}), but this behavior cannot be
detected a priori.
In order to construct SDP hierarchies  with guaranteed finite convergence behavior,
 some authors introduced redundant constraints in the hierarchies.
The work presented in \cite{DNS06} investigates unconstrained problems and the convergence of
SDP hierarchies over the variety of critical points,
while \cite{DNP07} considers more generally KKT ideals. 
The recent work \cite{ABM14} extends these results further and propose a relaxation
which is either exact or which detects in finitely many steps the absence of ``KKT minimizers''
 \cite[Th.~6.3]{ABM14}.

A drawback of this method is that it fails whenever optimal solutions of \cref{pb:parametric}
do not satisfy KKT conditions\footnote{Abril Bucero and Mourrain gave
hints to deal with such a situation, but at the expense of an increasing complexity
in the construction of the hierarchies.}.
\Cref{cor:diagonal-perturb} shows that this issue is only a concern for finitely many
values of the perturbation parameter $\prtb$ in \hyperref[pb:parametric]{$\PNLP$} 
and that the relaxation remains exact outside of this finite set. We  point out that the constructions presented in \cite{DNS06,DNP07}, similar in their
approach, require much stronger assumptions on the constraint ideal than the one we propose.

We now explain these facts;
\cref{sec:polynomials} contains the basic notation/definition used below.
We first describe the polynomial problem from which the relaxation in \cite{ABM14}
is constructed.
Let $\prtb \in \R$ be such that $C_\prtb = [g_1 \leq \prtb,\dots,g_m \leq \prtb]$ is nonempty.
The Lagrangian associated with \cref{pb:parametric} is defined for
$x \in \R^n$ and $\lambda \in \R^m$ by
\[
  L^\prtb(x,\lambda) := f(x) + \sum_{i = 1}^m \lambda_i \, ( g_i(x) - \prtb ) .
\]
Then we introduce the KKT ideal defined on $\R[x,\lambda]$ by
\[
  I_\KKT^\prtb := \left\langle \frac{\partial L^\prtb}{\partial x_1},\dots,
  \frac{\partial L^\prtb}{\partial x_n},
  \lambda_1 \, (g_1-\prtb),\dots,\lambda_m \, (g_m-\prtb) \right\rangle .
\]
Let $\{h^\prtb_1,\dots,h^\prtb_r\} \subset \R[x]$ be a generating family of the ideal
$I_\KKT^\prtb \cap \R[x]$:
\[
  \langle h^\prtb_1,\dots,h^\prtb_r \rangle = I_\KKT^\prtb \cap \R[x] .
\]
Note that such a family can be obtained by computing a Gr\"obner basis of $I_\KKT^\prtb$,
see \cite{CLO15}.
Adding these redundant constraints to \cref{pb:parametric} yields the following
polynomial problem.
\begin{equation}
  \label[prb]{pb:KKT}
  \tag{$\PKKT$}
  \begin{aligned}
    \text{minimize} & \enspace f(x) \\
    \text{subject to} & \enspace g_1(x) \leq \prtb, \dots, g_m(x) \leq \prtb  ,  \\
    & \enspace h^\prtb_1(x) = 0, \dots, h^\prtb_r(x) = 0 .
  \end{aligned}
\end{equation}
Observe that any minimizer of \cref{pb:parametric} that is also a KKT point
is a minimizer of \cref{pb:KKT}.
Hence, if the Mangasarian-Fromovitz constraint qualification holds throughout $C_\prtb$, then
solving the former problem boils down to solving the latter. 

We next introduce the SDP relaxation hierarchies proposed in \cite{ABM14} to solve
\labelcref{pb:KKT}.
For $k \in \N$, the primal is given by
\begin{multline}
  \label[prb]{pb:primal-relaxation}
    p^\prtb_k = \inf \big\{ \Lambda(f) \mid \enspace
    \Lambda \in (\R_{2k}[x])^* , \; \Lambda(1) = 1 , \\
    \Lambda(p) \geq 0 , \enspace \forall p \in \langle h^\prtb_1,\dots,h^\prtb_r \rangle_{2k} +
    \mathfrak{P}_k (\prtb-g_1,\dots,\prtb-g_m) \big\} ,
\end{multline}
and the dual problem is
\begin{equation}
  \label[prb]{pb:dual-relaxation}
  d^\prtb_k = \sup \big\{ \gamma \in \R \mid \enspace
  f - \gamma \in \langle h^\prtb_1,\dots,h^\prtb_r \rangle_{2k} +
  \mathfrak{P}_k (\prtb-g_1,\dots,\prtb-g_m) \big\} ,
\end{equation}
where the notation for the truncated ideal $\langle \cdot \rangle_{2k}$ and
the truncated preordering $\mathfrak{P}_k$ is detailed in \cref{sec:polynomials}.
Let us mention however that the dual relaxation hierarchy is based on an SOS representation
of nonnegative polynomials which uses a Schm\"udgen-type certificate.
But contrary to Schm\"udgen's Positivstellensatz \cite[Cor.~3]{Sch91}, compactness is
not required here.

A straightforward combination of \cref{cor:diagonal-perturb} and \cite[Th.~6.3]{ABM14}
leads to the following.

\begin{proposition}
  Let $f, g_1, \dots, g_m: \R^n \to \R$ be polynomial functions such that
  $C_0 = [g_1 \leq 0, \cdots, g_m \leq 0]$ is nonempty.
  Then, for all parameters $\prtb \geq 0$, except for a finite number of them,
  one of the following assertions holds:
  \begin{enumerate}
    \item the relaxations \labelcref{pb:primal-relaxation} and \labelcref{pb:dual-relaxation} of
      \cref{pb:KKT} are exact and provide the value of
      \labelcref{pb:parametric}, i.e.,
      $\val \PNLP = d^\prtb_k = p^\prtb_k$ for all $k$ large enough\footnote{The result of
        Abril Bucero and Mourrain is actually more precise and establishes a link between
        the minimizers of \cref{pb:primal-relaxation} and
        the ones of \cref{pb:parametric}.
        We refer the reader to \cite[Th.~6.3]{ABM14} for a comprehensive presentation.};
    \item for $k$ large enough, the feasible set of
      \cref{pb:primal-relaxation} is empty and
    \labelcref{pb:parametric} has no minimizer.
  \end{enumerate}
\end{proposition}

\appendix
\section{Reminder on semi-algebraic and tame geometry}
\label{sec:tame-geometry}

We recall here the basic results of tame geometry that we use in the present work.
Some references on this topic are \cite{vdDM96,vdD98,Cos99}.

\begin{definition}[see {\cite[Def.~1.4]{Cos99}}]
  An {\em o-minimal} structure on $(\R,+,\cdot)$ is a sequence of Boolean algebras
  $\mathcal{O} = (\mathcal{O}_p)_{p \in \N}$ where each $\mathcal{O}_p$ is a family of
  subsets of $\R^p$ and such that for each $p \in \N$:
  \begin{enumerate}
    \item if $A$ belongs to $\mathcal{O}_p$, then $A \times \R$ and $\R \times A$
      belong to $\mathcal{O}_{p+1}$;
    \item if $\pi: \R^{p+1} \to \R^p$ is the canonical projection onto $\R^p$ then,
      for any $A \in \mathcal{O}_{p+1}$, the set $\pi(A)$ belongs to $\mathcal{O}_p$;
      \label{it:algebraic}
    \item $\mathcal{O}_p$ contains the family of real algebraic subsets of $\R^p$, that is,
      every set of the form $\{ x \in \R^p \mid g(x) = 0 \}$
      where $g: \R^p \to \R$ is a polynomial function;
    \item the elements of $\mathcal{O}_1$ are exactly the finite unions of points and intervals.
  \end{enumerate}
\end{definition}

A subset of $\R^p$ which belongs to an o-minimal structure $\mathcal{O}$ is said to be
{\em definable} (in $\mathcal{O}$).
A function $f: A \subset \R^p \to \R^q$ or a set-valued mapping $F: A \subset \R^p \toto \R^q$
is said to be definable in $\mathcal{O}$ if its graph is definable (in $\mathcal{O}$)
as a subset of $\R^p \times \R^q$.

\begin{example}
  The simplest (and smallest) o-minimal structure is given by the class $\mathcal{SA}$ of
  real {\em semi-algebraic} objects.
 A set $A \subset \R^p$ is called semi-algebraic if it is of the form 
  $A = \bigcup_{j=1}^l \bigcap_{i=1}^k \{x \in \R^p \mid g_{ij}(x) < 0, \; h_{ij}(x) = 0 \}$
  where the functions $g_{ij}, h_{ij}: \R^p \to \R$ are polynomial functions.
  The fact that $\mathcal{SA}$ is an o-minimal structure relies mainly on the Tarski-Seidenberg
  principle (see \cite{BR90}) which asserts that \cref{it:algebraic} holds true in this class.

  Other examples like globally subanalytic sets or sets belonging to the $\log$-$\exp$
  structure provide a vast field of sets and functions that are of primary
  importance for optimizers.
  We will not give proper definitions of these structures in this paper, but the interested
  reader may consult \cite{vdDM96} or \cite{BDL06,Iof07,BDL09} for optimization
  oriented subjects.
\end{example}

In this paper, we shall essentially use the classical results listed hereafter.
In the remainder of this subsection, we fix an o-minimal structure $\mathcal{O}$
on $(\R,+,\cdot)$.

\begin{proposition}[stability results]
  Let $A \subset \R^p$ and $g: A \to \R^p$ be definable objects.
  \begin{itemize}
    \item If $B \subset A$ is a definable set, then $g(B)$ is definable.
    \item If $C \subset \R^q$ is a definable set, then $g^{-1}(C)$ is definable.
    \item If $A$ is open and $g$ is differentiable, then its derivative is definable.
  \end{itemize}
\end{proposition}

\begin{monotonicity}
  \label{lem:monotonicity}
  Let $f: I \subset \R \to \R$ be a definable function and let $k \in \N$.
  Then there exists a finite partition of $I$ into $p$ disjoint intervals $I_1,\dots,I_p$,
  such that the restriction of $f$ to each nontrivial interval $I_j$, $j \in \{1,\dots,p\}$,
  is $\mathcal{C}^k$ and either constant or strictly monotone. 
\end{monotonicity}

\begin{definable-choice}
  \label{lem:definable-choice}
  Let $A \subset \R^p \times \R^q$ be a definable set and let $\pi: \R^p \times \R^q \to \R^p$
  be the canonical projection onto $\R^p$.
  Then there exists a definable function $f: \pi(A) \to \R^q$ such that $\graph f \subset A$.
\end{definable-choice}

Note that an equivalent formulation of the latter result can be stated in terms of selection:
if $F: \R^p \toto \R^q$ is a definable set-valued mapping, then there exists
a definable function $f: \dom F \to \R^q$ such that $\graph f \subset \graph F$.

\begin{curve-selection}
  \label{lem:curve-selection}
  Let $A \subset \R^p$ be a definable set, $x$ be an element of $\operatorname{cl}(A)$,
  the topological closure of $A$, and let $k \in \N$ be a fixed integer.
  Then there exists a $\mathcal{C}^k$ definable path $\gamma: [0,1) \to \R^p$
  such that $\gamma(0) = x$ and $\gamma((0,1)) \subset A$.
\end{curve-selection}

\section{Relaxation in polynomial programming: definitions and notation}
\label{sec:polynomials}

By $\R[x]$ we denote the ring of real polynomials in the variable $x = (x_1,\dots,x_n)$.
For any $k \in \N$, we denote by $\R_k[x]$ the space of real polynomials
whose degree is bounded by $k$, and we denote by $(\R_k[x])^*$ its dual space.

A polynomial $p \in \R[x]$ is a sum of squares (SOS) if $p$ can be written as
$p = \sum_{i \in I} p_i^2$ for some finite family of polynomials $(p_i)_{i \in I} \subset \R[x]$.
Denote by $\Sigma[x]$ the space of SOS polynomials.

Given any integer $k \in \N$ and any finite family $\{ p_1,\dots,p_m \} \subset \R[x]$ of
polynomials, the {\em $k$-truncated ideal} on $\R[x]$ generated by this family is
the subset of $\R[x]$ defined by
\[
  \langle p_1,\dots,p_r \rangle_k := \bigg\{ \sum_{i = 1}^m q_i \, p_i \mid
  q_i \in \R[x], \; \deg(q_i \, p_i) \leq k, \; i = 1,\dots,m \bigg\} ,
\]
where $\deg(p)$ denotes the degree of any polynomial $p \in \R[x]$.
The ideal generated by the family $\{ p_1,\dots,p_m \}$ is denoted and defined
in a similar way but with no condition required on the degree of the polynomials.

For a set $I \subset \{1,\dots,m\}$, we denote by $p_I \in \R[x]$ the polynomial defined by
$p_I := \prod_{i \in I} p_i$, with the convention that $p_\emptyset = 1$.
Then we define the {\em $k$-truncated preordering} of $\{ p_1,\dots,p_m \}$ by
\[
  \mathfrak{P}_k(p_1,\dots,p_m) := \bigg\{ \sum_I q_I \, p_I \mid
  q_I \in \Sigma[x], \; \deg(q_I \, p_I) \leq 2k, \; \forall I \subset \{1,\dots,m\} \bigg\} .
\]

 \section*{Acknowledgments}
 The authors thanks J.~Pang, H.~Frankowska, T.~S.~Pham, J.-B.~Lasserre for their useful comments.

\bibliographystyle{alpha}
\bibliography{references}

\end{document}